\documentclass[onefignum,onetabnum]{siamart220329}

\usepackage{algorithmic}
\usepackage{geometry}

\usepackage{microtype}
\usepackage{graphicx}
\usepackage{booktabs} 
\usepackage{microtype}
\usepackage{graphicx}
\usepackage{multirow}
\usepackage{booktabs} 
\usepackage{hyperref}
\usepackage{subcaption}
\usepackage{pgfplots}
\usepackage{pgfplotstable}
\pgfplotsset{compat=newest}
\usetikzlibrary{spy}
\usepackage{hyperref}

\usepackage{pifont}
\usepackage{amsmath}
\usepackage{amssymb}
\usepackage{mathtools}
\usepackage{tabularx}

\numberwithin{equation}{section}

\newtheorem{assumption}[theorem]{Assumption}

\ifpdf
  \DeclareGraphicsExtensions{.eps,.pdf,.png,.jpg}
\else
  \DeclareGraphicsExtensions{.eps}
\fi
\allowdisplaybreaks[1]

\headers{A Single-loop Stochastic Riemannian ADMM}{Jin and Deng and Wang}

\title{A Single-loop Stochastic Riemannian ADMM for Nonsmooth Composite Optimization}

\author{
Jiachen Jin\thanks{College of Science,  National University of Defense Technology, Changsha, 410073,
China (\email{jinjiachen@nudt.edu.cn}, \email{freedeng1208@gmail.com}, \email{wanghongxia@nudt.edu.cn})} 
\and Kangkang Deng\footnotemark[1] \and Hongxia Wang\footnotemark[1] \thanks{Corresponding author.}
}

\usepackage{amsopn}

\ifpdf
\hypersetup{
  pdftitle={A Single-loop Stochastic Riemannian ADMM for Nonsmooth Composite Optimization},
  pdfauthor={Jiachen Jin}
}
\fi

\allowdisplaybreaks[2]
\begin{document}

\maketitle

\begin{abstract}
We study a class of nonsmooth composite stochastic optimization on Riemannian manifolds, where the objective is the sum of an expectation function and a nonsmooth regularizer. These types of problems appear widely in various application fields, such as machine learning. Although operator-splitting methods naturally exploit this separable structure, existing Riemannian variants primarily target deterministic problems, and stochastic extensions are limited to more restrictive settings. 
In this work, we propose a momentum-based adaptive Riemannian stochastic alternating direction method of multipliers (MARS-ADMM), which combines a recursive variance-reduced estimator for the expectation component with a prediction-correction update for the nonsmooth block. This yields a single-loop algorithm requiring only two proximal updates, one Riemannian gradient evaluation, and $\mathcal{O}(1)$ stochastic gradient samples per iteration. 
Under standard assumptions, we prove that MARS-ADMM attains an $\epsilon$-KKT point with an oracle complexity of $\mathcal{O}(\epsilon^{-3})$, improving upon the previously best-known rate of $\widetilde{\mathcal{O}}(\epsilon^{-3.5})$ for stochastic Riemannian primal-dual methods. This complexity also matches the best-known bounds in deterministic nonsmooth Riemannian optimization, demonstrating that deterministic-level accuracy can be achieved using only constant-size stochastic samples. Numerical experiments on two types of test problems reveal promising performances of the proposed algorithm. To the best of our knowledge, MARS-ADMM is the first stochastic Riemannian ADMM with provable optimal complexity guarantees. 
\end{abstract}
\begin{keywords}
Stochastic optimization, Riemannian optimization, composite optimization, ADMM, variance reduction, oracle complexity
\end{keywords}

\begin{AMS}
 65K05, 65K10, 90C15, 90C26
\end{AMS}

\section{Introduction}
We consider the nonsmooth stochastic (online) optimization problem on manifolds:
\begin{equation}\label{op}
	\min_{x\in\mathcal{M}} \{F(x) :=\mathbb{E}_{\xi}[f(x,\xi)]\} + g(Ax),
\end{equation}
where $\mathcal{M}$ is a Riemannian submanifold of $\mathbb{R}^n$, $\xi$ is a random variable in the probability space $\Xi$, $\mathbb{E}_{\xi}$ is the expectation with respect to the random variable $\xi$, $f(\cdot, \xi): \mathbb{R}^n \to \mathbb{R}$ is a smooth function for each $\xi$, $g: \mathbb{R}^m \to \mathbb{R}$ is a convex and possibly nonsmooth function, and $A\in \mathbb{R}^{m\times n}$. In practice, calculating the expectation may be computationally expensive, or the distribution of $\xi$ may be unknown.  Such problems arise in numerous machine learning and signal processing applications, including online sparse principal component analysis \cite{wang2016online}, dictionary recovery \cite{sun2016complete} and tensor factorization \cite{ishteva2011best}, where the data are inherently stochastic and the parameter space exhibits nonlinear geometric constraints.

Owing to the separable structure of the objective function, operator-splitting methods have become a standard and powerful tool for solving problem~\eqref{op}. For the general composite case, a standard reformulation introduces an auxiliary variable $y = Ax$, resulting in a block-separable problem:
\begin{equation}\label{p}
	\min_{x\in\mathcal{M},y} F(x) + g(y),~\text{s.t.}~Ax=y.
\end{equation} 
In the deterministic setting, a broad class of operator-splitting algorithms has been extensively investigated; see, for example, \cite{lai2014splitting,kovnatsky2016madmm,deng2024oracle,deng2023manifold,dengadaptive,xu2024riemannian,xu2025oracle,yuan2024admm,li2025riemannian,el2025convergence}. Among them, the Riemannian alternating direction method of multipliers (ADMM) has attracted particular attention due to its intrinsic decoupling property and strong empirical performance. 
However, extending Euclidean ADMM to Riemannian manifolds is highly nontrivial. Due to the inherent nonconvexity induced by manifold constraints, several early attempts failed to guarantee convergence \cite{lai2014splitting,kovnatsky2016madmm}. To address these challenges, several variants of Riemannian ADMM have been recently proposed \cite{yuan2024admm,li2025riemannian,dengadaptive}, together with rigorous convergence analyses. Nevertheless, existing Riemannian ADMM algorithms are restricted to deterministic settings.

In contrast, the available stochastic primal-dual methods apply to more restrictive structures. \cite{zhang2020primal} proposed a stochastic primal-dual ADMM algorithm on manifolds; however, the resulting subproblems are computationally challenging and the analysis relies on a strong structural assumption. When specialized to the general structure of problem \eqref{p}, the stated bound is $\mathcal{O}(\epsilon^{-4})$. \cite{deng2024oracle} developed a stochastic manifold augmented Lagrangian method with an oracle complexity of $\widetilde{\mathcal{O}}(\epsilon^{-3.5})$.  However, this algorithm uses a double-loop framework, which significantly increases computational cost. On the other hand, when manifold constraints are treated as general nonlinear constraints, the accelerated successive convex approximation algorithm in \cite{idrees2025constrained} has a near-optimal oracle complexity of $\widetilde{\mathcal{O}}(\epsilon^{-3})$, and the momentum-based linearized augmented Lagrangian method in \cite{shi2025momentum} achieves an optimal rate of $\mathcal{O}(\epsilon^{-3})$. These results reveal a theoretical complexity gap between stochastic Riemannian primal-dual methods for problem \eqref{p} and state-of-the-art stochastic constrained algorithms. This observation naturally leads to the following fundamental question:
\begin{quote}
	\emph{Can we design a stochastic Riemannian primal-dual algorithm that closes this complexity gap?}
\end{quote}

This paper provides an affirmative answer by proposing an momentum-based adaptive stochastic Riemannian ADMM (MARS-ADMM). The main contributions are summarized as follows:

\begin{itemize}
	\item We propose MARS-ADMM, the first stochastic Riemannian ADMM method for solving problem \eqref{p}. 
	The approach employs a prediction-correction Gauss-Seidel sweep and orders the blocks as $y \to x\to y\to\lambda$. The first $y$-update provides the augmented gradient predictor for the linearized manifold block. The second $y$-update corrects the nonsmooth block at the updated $x$. 
	Between these updates, a single stochastic Riemannian gradient step based on momentum is computed for the $x$-block.
	This two-stage $y$-update scheme mitigates the inherent information lag of Riemannian ADMM and enables us to propose a single-loop framework, as opposed to the double-loop algorithm in \cite{deng2024oracle}. 
	
	\item Our approach differs from existing methods in several key aspects. The prediction-correction mechanism only acts on the nonsmooth block to compensate for the inexact $x$-subproblem caused by the Riemannian constraint. This is fundamentally different from the prediction-correction ADMM methods in \cite{he2012alternating,chang2018convergent}, which use a full ADMM prediction with a contraction correction to ensure convergence for convex problems. 
	The Riemannian primal stepsize is data-dependent and self-normalized in the spirit of STORM$^+$ in \cite{levy2021stormplus}. The momentum estimator is only applied to the smooth stochastic component; therefore, its recursion does not inherit the increasing penalty. 
	Additionally, the dual stepsize retains the penalty $\rho_k$ to maintain the current full residual rather than the vanishing coefficient in \cite{dengadaptive} to reflect constraint violations in a timely manner. 
	
	\item Under standard assumptions, MARS-ADMM achieves an $\epsilon$-KKT point with an oracle complexity of $\mathcal{O}(\epsilon^{-3})$, improving upon the best-known complexity bounds of $\tilde{\mathcal{O}}(\epsilon^{-3.5})$ for stochastic Riemannian primal-dual algorithms \cite{deng2024oracle}. 
	This complexity also matches the best-known bounds for deterministic nonsmooth Riemannian optimization \cite{beck2023dynamic,dengadaptive}, demonstrating that deterministic-level accuracy can be attained using only constant-size stochastic samples.   
	The key to establishing the aforementioned convergence result is adopting a self-normalized cumulative error estimation and regarding the multiplier update as a contraction toward the dual maximizer of a moving Moreau envelope.
\end{itemize}

A comparison between our algorithm and several stochastic methods is shown in Table \ref{table1}. 
We do not list the R-ProxSPB method \cite{wang2022riemannian} since each of its oracle calls involves a semismooth Newton subroutine whose overall complexity is unclear. 

\begin{table*}[h]
	\centering
	\resizebox{1\columnwidth}{!}{
		\begin{tabular}{cccc}
			\hline
			Algorithms & Constraints  &  Loop & Complexity \\ \hline
			LCSPG \cite{boob2025level} & $c(x)+h(x)\leq 0$ &  double &  $\mathcal{O}(\epsilon^{-4})$ \\
			CoSTA \cite{idrees2025constrained} & $c(x)\leq 0,d(x)\leq 0$ &  single &  $\tilde{\mathcal{O}}(\epsilon^{-3})$\\
			MLALM \cite{shi2025momentum} & $c(x)\leq 0$ & single &   $\mathcal{O}(\epsilon^{-3})$\\
			\hline
			Proximal ADMM \cite{zhang2020primal} & compact submanifold &  double  &  $\mathcal{O}(\epsilon^{-4})$\\
			subgradient \cite{li2021weakly} & Stiefel manifold &  single  &  $\mathcal{O}(\epsilon^{-4})$\\
			smoothing \cite{peng2023riemannian} & 
			compact submanifold &  single &   $\mathcal{O}(\epsilon^{-5})$\\
			StoManIAL \cite{deng2024oracle} & compact submanifold &  double &  $\tilde{\mathcal{O}}(\epsilon^{-3.5})$ \\
			\textbf{MARS-ADMM (this paper)}  & compact submanifold & single &  $\mathcal{O}(\epsilon^{-3})$\\ 
			\hline
	\end{tabular}}
	\caption{Comparison across related stochastic composite and constrained problem classes, where $\tilde{\mathcal{O}}$ denotes the asymptotic upper bound that ignores log factors.}
	\label{table1}
\end{table*}

\subsection{Literature Review} 
\paragraph{Riemannian ADMM} 
The ADMM algorithm is a widely used operator-splitting method for solving nonsmooth composite optimization problems. However, it has primarily been studied in deterministic settings. One of the earliest attempts to decouple manifold constraints from nonsmooth terms was the splitting orthogonality constraints (SOC) method in \cite{lai2014splitting}, which involved solving an unconstrained subproblem and a manifold projection subproblem. Manifold ADMM (MADMM) expanded on this concept in \cite{kovnatsky2016madmm} by treating the subproblem as a smooth Riemannian optimization problem and applying a proximal operator to the nonsmooth term in the subproblem. Despite promising empirical results, both the SOC and MADMM methods lack theoretical guarantees. 
In Euclidean space, full residual correction is often controllable by using a full-rank condition to express multiplier differences through primal optimality of the smooth subproblem. However, the changing tangent spaces invalidate that argument on a manifold. 
Subsequent work in \cite{li2025riemannian} applied Moreau smoothing to the nonsmooth term to provide theoretical guarantees. This strategy was further refined with adaptive smoothing in \cite{yuan2024admm}. However, such smoothing techniques essentially convert the original nonsmooth problem into a smooth one. This alters the nature of the ADMM algorithm because the smoothed approximation no longer contains difficult coupling and can be addressed directly using gradient-based methods \cite{beck2023dynamic}. An adaptive Riemannian ADMM was proposed in \cite{dengadaptive} to solve the original nonsmooth problem directly. This method attenuates the multiplier residual coefficient to make the total movement of the multipliers summable. The stepsize eventually becomes very small, causing the multiplier to stagnate even when the primal residual is not eliminated. While some nonconvex Riemannian ADMM methods \cite{lu2018nonconvex,chen2025non} have convergence guarantees for specific smooth applications, they are not suitable for nonsmooth problems. In addition to the aforementioned deterministic methods, a finite-sum variant with complexity $\mathcal{O}(\epsilon^{-2})$ was proposed in \cite{zhang2020primal} for a specific stochastic problem that excludes the last block variable from nonsmooth terms and manifold constraints. This contrasts sharply with problem \eqref{p}. Relaxing this condition increases the complexity to $O(\epsilon^{-4})$.

\paragraph{Stochastic Riemannian optimization} 
There are several methods for directly solving the Riemannian optimization problem \eqref{op}. When the nonsmooth term $g$ vanishes, Riemannian stochastic gradient descent (R-SGD) becomes an efficient algorithm for solving smooth stochastic problems \cite{bonnabel2013stochastic}. Furthermore, R-SVRG \cite{zhang2016riemannian}, R-SRG \cite{kasai2018riemannian}, R-SPIDER \cite{zhang2018r} and R-SRM \cite{han2021riemannian} use variance reduction techniques to improve R-SGD's convergence rate and achieve optimal oracle complexity of $\mathcal{O}(\epsilon^{-3})$. However, the nonsmooth term g makes developing stochastic algorithms for problem \eqref{op} challenging. \cite{li2021weakly} designed a Riemannian stochastic subgradient method for weakly convex problems on the Stiefel manifold with complexity $\mathcal{O}(\epsilon^{-4})$. \cite{wang2022riemannian} proposed Riemannian stochastic proximal-gradient methods for nonsmooth optimization over the Stiefel manifold. While these methods achieve order $\mathcal{O}(\epsilon^{-3})$, each oracle call involves a subroutine whose overall complexity is unspecified. Leveraging Moreau smoothing, \cite{peng2023riemannian} presented a Riemannian stochastic smoothing algorithm with complexity of order $\mathcal{O}(\epsilon^{-5})$. More recently, \cite{deng2024oracle} developed a stochastic augmented Lagrangian method that incorporates Riemannian recursive momentum and has improved complexity of order $\widetilde{\mathcal{O}}(\epsilon^{-3.5})$.

\paragraph{Nonlinear constrained stochastic optimization} 
When the manifold constraints are expressed as equality constraints $c(x)=0$, problem \eqref{op} can be viewed as a constrained optimization problem with a nonconvex, nonsmooth objective. When the nonsmooth term $g$ is absent, the single-loop momentum-based method in \cite{alacaoglu2024complexity} achieved complexity of order $\widetilde{\mathcal{O}}(\epsilon^{-4})$. The method in \cite{lu2024variance} used truncated recursive momentum and increasing penalty parameters to attain complexity of order $\widetilde{\mathcal{O}}(\epsilon^{-3})$. Recently, \cite{cui2025exact} studied an exact-penalty method within a double-loop framework and established an oracle complexity of $\mathcal{O}(\epsilon^{-3})$. For stochastic composite problems with general constraints $c(x)+h(x)\leq0$, the level-constrained stochastic proximal-gradient method in \cite{boob2025level} used an increasing constraint-level scheme to guarantee an oracle complexity of order $\mathcal{O}(\epsilon^{-4})$ under suitable constraint qualifications. \cite{idrees2025constrained} proposed the CoSTA algorithm, which introduces recursive momentum updates to achieve an almost optimal oracle complexity of $\widetilde{\mathcal{O}}(\epsilon^{-3})$ under adaptive stepsizes and Lipschitz smoothness assumptions. Inspired by the same variance-reduced scheme, MLALM \cite{shi2025momentum} attained a complexity of order $\mathcal{O}(\epsilon^{-4})$, which improves to $\mathcal{O}(\epsilon^{-3})$ when the initial point is nearly feasible.

\subsection{Notation}
The $n$-dimensional Euclidean space is denoted by $\mathbb{R}^n$, and its inner product by $\langle \cdot,\cdot \rangle$. 
For vectors, $\|\cdot\|$ denotes the Euclidean norm; for a matrix $A$, $\|A\|=\max_{\|x\|=1}\|Ax\|$. 
The symbol $\|\cdot\|_1$ denotes the vector $\ell_1$ norm and, when applied to a matrix, the entrywise norm $\|A\|_1=\sum_{ij}|A_{ij}|$ rather than the induced matrix $1$-norm. 
The distance from $x$ to a set $\mathcal{Q}$ is ${\rm dist}(x,\mathcal{Q})=\inf_{y\in\mathcal{Q}}\|x-y\|$. 
The symbols $\nabla F(x)$ and ${\rm grad}\, F(x)$ denote the Euclidean and Riemannian gradients of $F(x)$, respectively. 
For a batch \(\mathcal S\) of size $|\mathcal S|$, write the stochastic gradient
$\operatorname{grad}f_{\mathcal S}(x)=\frac{1}{|\mathcal S|}\sum_{\xi\in\mathcal S}\operatorname{grad}f(x,\xi)$. 

\subsection{Organization}
The outline of this paper is as follows. 
Section \ref{sec2} introduces the geometric, proximal, and stationarity notions.
Section \ref{sec3} presents MARS-ADMM, and Section \ref{sec4} establishes its convergence.
Section \ref{sec5} reports numerical experiments, and Section \ref{sec6} concludes the paper.

\section{Preliminaries}\label{sec2} 
This section begins by presenting the basic setup and mathematical tools of Riemannian optimization. Further details can be found in \cite{absil2009optimization,boumal2023introduction}. Next, we state the proximal operator and define the optimality measure.

\subsection{Riemannian Optimization}
An $n$-dimensional smooth manifold $\mathcal{M}$ is a topological space equipped with a smooth structure, where each point has a neighborhood that is a diffeomorphism of $\mathbb{R}^n$. 
\begin{definition}[Tangent Space]
	Consider a manifold $\mathcal{M}$ embedded in Euclidean space $\mathbb{R}^n$. For any $x\in\mathcal{M}$, the tangent space $T_x\mathcal{M}$ at $x$ is a linear subspace that consists of the derivatives of all differentiable curves on $\mathcal{M}$ passing through $x$:
	\begin{equation}\label{def:TS}
		T_x\mathcal{M}=\{\gamma'(0):\gamma(0)=x,\gamma([-\varepsilon,\varepsilon])\subset\mathcal{M}~{\rm for~some~\varepsilon>0},\gamma~{\rm is~differentiable}\}.
	\end{equation}
\end{definition}
If $\mathcal{M}$ is a Riemannian submanifold of Euclidean space $\mathbb{R}^n$, the inner product is defined as the Euclidean inner product. 
The Riemannian gradient ${\rm grad}\, F(x)\in T_x\mathcal{M}$ is a tangent vector satisfying
\[\langle v,{\rm grad}\, F(x)\rangle=\left.\frac{d(F(\gamma(t)))}{dt}\right|_{t=0},\forall v\in T_x\mathcal{M},\]
where $\gamma(t)$ is a curve as described in \eqref{def:TS}. If $\mathcal{M}$ is an embedded Riemannian submanifold, then ${\rm grad}\, F(x) = \mathcal{P}_{T_x\mathcal{M}}(\nabla F(x))$, where $\mathcal{P}_{T_x\mathcal{M}}$ denotes the orthogonal projection onto $T_x\mathcal{M}$. 
The retraction turns an element of $T_x\mathcal{M}$ into a point in $\mathcal{M}$.
\begin{definition}[Retraction]\label{def-retr}
	A retraction on a manifold $\mathcal{M}$ is a smooth mapping $\mathcal{R}:T\mathcal{M}\rightarrow \mathcal{M}$ with the following properties. Let $\mathcal{R}_x:T_x\mathcal{M} \rightarrow \mathcal{M}$ be the restriction of $\mathcal{R}$ at $x$. It satisfies
	$\mathcal{R}_x(0_x) = x$ and $d\mathcal{R}_x(0_x) = id_{T_x\mathcal{M}}$, where $0_x$ is the zero element of $T_x\mathcal{M}$ and $id_{T_x\mathcal{M}}$ is the identity mapping on $T_x\mathcal{M}$.
\end{definition}

The following uniform retraction inequalities hold on a compact submanifold.
\begin{lemma}\cite{boumal2019global}\label{proposi:ret-lips}
	Let $\mathcal{R}$ be a retraction operator on a compact submanifold $\mathcal{M}$. Then, there exist two constants $\varrho,\beta>0$ such that for all $x\in \mathcal{M},u \in T_{x}\mathcal{M}$, we have
	\[\|\mathcal{R}_x(u) - x\| \leq \varrho \|u\|,~\|\mathcal{R}_x(u) - x-u\| \leq \beta \|u\|^2.\]
\end{lemma}

\begin{lemma}\cite{boumal2019global}\label{lem:Lpb}
	Let $\mathcal{M}$ be a compact Riemannian submanifold of $\mathbb{R}^n$ and let $\mathcal{R}$ be a retraction on $\mathcal{M}$. If $F$ has a Lipschitz-continuous gradient on ${\rm conv}(\mathcal{M})$, then there is a constant $L_F^R\geq0$ such that, for all $x\in\mathcal M$ and $\omega\in T_x\mathcal M$,
	\[\left|F(\mathcal R_x(\omega))-F(x)-\langle \omega,{\rm grad}\, F(x)\rangle\right|\leq\frac{L_F^R}{2}\|\omega\|^2.\]
\end{lemma}

The ambient tangent projections also vary uniformly smoothly.
\begin{lemma}[Tangent-projection Lipschitzness]\label{lem:projector-lipschitz}
	There is a constant \(L_{\mathcal P}\), depending only on the compact smooth embedded submanifold \(\mathcal M\), such that, for all $x,x'\in\mathcal M$,
	\[
	\|\mathcal P_{T_x\mathcal M}-\mathcal P_{T_{x'}\mathcal M}\|
	\leq L_{\mathcal P}\|x-x'\|.
	\]
\end{lemma}
\begin{proof}
	For a smooth embedded submanifold, \(x\mapsto\mathcal P_{T_x\mathcal M}\) is a smooth matrix-valued map. Therefore, it is locally Lipschitz in embedded coordinates. Compactness provides a uniform local Lipschitz constant. For pairs outside a fixed neighborhood of the diagonal in \(\mathcal M\times\mathcal M\), the denominator \(\|x-x'\|\) has a positive minimum, whereas the difference of two orthogonal projectors is bounded. Combining the local and off-diagonal bounds proves the claim.
\end{proof}
For $u\in T_x\mathcal{M}$, let $c_{x,u}(t)=\mathcal{R}_x(tu)$, $t\in[0,1]$. Throughout the paper,
\[
\mathcal{T}_x^{\mathcal{R}_x(u)}:T_x\mathcal{M}\to T_{\mathcal{R}_x(u)}\mathcal{M}
\]
denotes parallel transport along this retraction curve. This convention fixes a single transport for both the algorithm and the stochastic smoothness assumption. Parallel transport is linear and isometric:
\[
\langle v,w\rangle
=\left\langle\mathcal{T}_x^{\mathcal{R}_x(u)}v,\mathcal{T}_x^{\mathcal{R}_x(u)}w\right\rangle
\quad\text{for all }v,w\in T_x\mathcal{M}.
\]

\subsection{Proximal Operator and Optimality Measure}
Let $g$ be an extended real-valued convex function. Its conjugate function is defined as 
\[g^*(z)=\sup_{x\in \operatorname{dom} g} [\langle z,x \rangle-g(x)].\]
The following lemma defines the proximal operator for $g$ and introduces its related properties.
\begin{lemma}\cite{beck2017first}\label{lem:prox}
	The Moreau envelope of $g$ with $\delta>0$ is given by
	\[M_g^\delta (y)=\min_{z\in\mathbb{R}^m}\left\{g(z)+\frac{1}{2\delta}\|z-y\|^2\right\}.\]
	The proximal operator is given by
	\[\operatorname{prox}_{\delta g}(y)=\arg\min_{z\in\mathbb{R}^m}\left\{g(z)+\frac{1}{2\delta}\|z-y\|^2\right\}.\]
	These functions have the following characteristics:
	\begin{enumerate}
		\item[(a)] (Nonexpansiveness) $\|\operatorname{prox}_{ g}(x)-\operatorname{prox}_{g}(y)\|\leq \|x-y\|.$
		
		\item[(b)] (Moreau decomposition) $\operatorname{prox}_{\delta g}(y)+\delta \operatorname{prox}_{\frac{g^*}{\delta}}(\frac{y}{\delta})=y$.
		
		\item[(c)] (Lipschitz smooth) $M_g^\delta (y)$ is $\frac{1}{\delta}$-smooth, and $\nabla M_g^\delta (y)=\frac{1}{\delta}(x-\operatorname{prox}_{\delta g}(y))$.
	\end{enumerate}
\end{lemma}

Finally, we give the optimality measure for problem \eqref{p}. Define the Lagrangian function
\[\mathcal L(x,y,\lambda)=F(x)+g(y)-\langle \lambda, Ax-y \rangle,\]
where $\lambda$ is called the Lagrange multiplier. 
A Karush-Kuhn-Tucker (KKT) triple $(x,y,\lambda)$ satisfies
\[
\mathcal P_{T_x\mathcal M}\big(\nabla F(x)-A^\top\lambda\big)=0,\quad
\partial g(y)+\lambda\in 0,\quad 
Ax-y=0.
\]
where $\partial\, g$ is the set of general subgradients of $g$, as denoted in \cite[Definition 8.3]{rockafellar1998variational}. 
The following definition measures all three conditions at one and the same (possibly randomized) triple.
\begin{definition}[$\epsilon$-KKT point]
	A possibly randomized triple $(\widehat x,\widehat y,\widehat\lambda)$ is an  $\epsilon$-KKT point of problem \eqref{p} if
	\[\begin{cases}
		\mathbb{E}\left[\left\|\mathcal{P}_{T_{\widehat x}\mathcal{M}}
		(\nabla F(\widehat x)-A^\top\widehat\lambda)\right\|\right]\leq\epsilon,\\
		\mathbb{E}\left[{\rm dist}(-\widehat\lambda,\partial g(\widehat y))\right]\leq\epsilon,\\
		\mathbb{E}\left[\|A\widehat x-\widehat y\|\right]\leq\epsilon.
	\end{cases}\]
	We call $(\widehat x,\widehat y,\widehat\lambda)$ an (expected) $\epsilon$-KKT triple.
\end{definition}

In this work, stochastic first-order oracles (SFO), proximal, and retraction costs are counted separately.
\begin{definition}
	Given a input $x$ and a sample $\xi$, one SFO evaluation returns $\nabla f(x,\xi)$; its orthogonal projection gives ${\rm grad}\, f(x,\xi)$. A proximal evaluation and a retraction evaluation are counted as separate operations.
\end{definition}

\section{Momentum-based Adaptive Riemannian Stochastic ADMM}\label{sec3}
We begin with the augmented Lagrangian function: 
\[
\mathcal L_{\rho}(x,y,\lambda)=F(x)+g(y)-\langle \lambda,Ax-y\rangle+\frac\rho2\|Ax-y\|^2.
\]
where $\rho>0$ is the penalty parameter. 
Directly applying ADMM with classical dual updates to problem \eqref{p} is problematic because the boundedness of dual iterates is difficult to analyze on a manifold. The projected $x$-optimality condition does not recover $A^\top\lambda$, and no rank condition is imposed on $A$. 
We retain the full current residual coefficient by damping only the incoming multiplier and use a Gauss-Seidel correction of the nonsmooth block so that the reported multiplier is associated with the final primal pair.

As shown in \cite{li2024stochastic,dengadaptive}, a Riemannian ADMM variant with a damped multiplier $\tau\lambda_k$ takes the form at $k$-iteration:
\[
\left\{
\begin{aligned}
	&x_{k+1} = \mathcal{R}_{x_k}\big(-\eta_k {\rm grad}_x \mathcal{L}_{\rho}(x_k,y_k,\tau\lambda_k)\big),\\
	&y_{k+1}=\arg\min_y \{\mathcal{L}_{\rho}(x_{k+1},y,\tau\lambda_k)\},\\
	&\lambda_{k+1}=\tau\lambda_k-\rho(Ax_{k+1}-y_{k+1}).
\end{aligned}
\right.
\]
where $\tau\in(0,1)$ is the damping parameter. 
Tracing the dependencies of the primal variables around the manifold update $x_{k+1}$ reveals that
\[
y_k \leftarrow (x_k,\lambda_{k-1}),\quad 
x_{k+1} \leftarrow (x_k,y_k,\lambda_k),\quad 
y_{k+1} \leftarrow (x_{k+1},\lambda_k).
\]
Thus $y_k$ encodes the response of the nonsmooth block $g(y)$ to the previous iterate $x_k$ and the multiplier $\lambda_{k-1}$. However, when updating $x_{k+1}$, the multiplier advances to $\lambda_k$, and the manifold point $x_k$ itself may have moved substantially. The augmented-gradient term in the $x$-update uses a stale $y_k$. 
In standard ADMM, the $x$-subproblem is solved exactly, which implicitly corrects the inconsistency over subsequent iterations. 
To achieve greater accuracy in a single Riemannian gradient step, we replace $y_k$ in the $x$-update with a fresh prediction computed from the current iterate $x_k$ and the current multiplier $\lambda_k$. 
The resulting dependency structure becomes
\[
\widehat y_{k+1} \leftarrow (x_k,\lambda_k),\quad
x_{k+1} \leftarrow (x_k,\widehat y_{k+1},\lambda_k),\quad
y_{k+1} \leftarrow (x_{k+1},\lambda_k).
\]
The predictor $\widehat y_{k+1}$ provides up-to-date augmented gradient information for the next single Riemannian gradient step. The corrector $y_{k+1}$ restores optimality with respect to the new manifold variable for the nonsmooth block.

To handle the stochastic term $F(x)=\mathbb{E}_{\xi}[f(x,\xi)]$, we use a variance-reduced stochastic estimator. The Riemannian gradient of $\mathcal{L}_{\rho_k}$ with respect to $x$ has the decomposition: 
\[
{\rm grad}_x \mathcal{L}_{\rho_k}(x_k,\widehat y_{k+1},\tau\lambda_k)
={\rm grad}\, F(x_k)+ \mathcal{P}_{T_{x_k}\mathcal{M}}(\rho_k A^\top( Ax_k - \widehat y_{k+1} - \tau\lambda_k /\rho_k)).
\]
Replacing only the stochastic part by a recursive estimator $v_k$, we form the full gradient approximation
\begin{equation}\label{eq:direction}
	\mathcal G_k=v_k+ \mathcal{P}_{T_{x_k}\mathcal{M}}\big(\rho_k A^\top( Ax_k - \widehat y_{k+1} - \tau\lambda_k /\rho_k)\big),
\end{equation}
where $v_k$ is updated recursively as:
\begin{align}
	v_k =&\alpha_k {\rm grad}\, f_{\mathcal{S}_k}(x_k)+ (1-\alpha_k)({\rm grad}\, f_{\mathcal{S}_k}(x_k)- \mathcal{T}^{x_k}_{x_{k-1}}({\rm grad}\, f_{\mathcal{S}_k}(x_{k-1})-v_{k-1}))\nonumber\\
	=&{\rm grad}\, f_{\mathcal{S}_k}(x_k)+ (1-\alpha_k)\mathcal{T}^{x_k}_{x_{k-1}}(v_{k-1} - {\rm grad}\, f_{\mathcal{S}_k}(x_{k-1})).\label{eq:storm}
\end{align} 
The parameter $\alpha_k\in(0,1]$ controls the blending of the current stochastic gradient and past estimators information.  
The complete method is summarized in Algorithm \ref{alg1}.

\begin{algorithm}[H]
	\caption{Momentum-based Adaptive Riemannian Stochastic ADMM (MARS-ADMM)}\label{alg1}
	\textbf{Require}: \(x_1\in\mathcal M\);
	\(\lambda_1=\operatorname{prox}_g(0)\);
	\(\tau\in(0,1)\); and parameters \(c_\rho,c_\eta,c_\alpha>0\).
	
	\begin{algorithmic}[1]
		\STATE Draw \(\mathcal S_1\) and set \(v_1=\operatorname{grad}f_{\mathcal S_1}(x_1)\).
		
		\FOR{\(k=1,\ldots,K\)}
		
		\STATE Update the predictor \(\widehat y_{k+1}\) via $\rho_k=c_\rho k^{1/3}$ and
		\begin{equation}\label{eq:y-predictor}
			\widehat y_{k+1}=\arg\min_y \mathcal L_{\rho_k}(x_k,y,\tau\lambda_k).
		\end{equation}
		
		\STATE Compute the stochastic gradient estimator $\mathcal{G}_k$ by \eqref{eq:direction} with
		\begin{equation}\label{eq:adaptive-schedules}
			\alpha_{k+1}=c_\alpha k^{-2/3},
			\quad
			\eta_k=\frac{c_\eta}{\left(k+\sum_{i=1}^{k}\frac{\|\mathcal G_i\|^2}{\alpha_{i+1}}\right)^{1/3}}.
		\end{equation}
		Update $x_{k+1}$ by
		\begin{equation}\label{eq:x-full}
			x_{k+1}=\mathcal R_{x_k}(-\eta_k\mathcal G_k).
		\end{equation}
		
		\STATE Correct the nonsmooth block by
		\begin{equation}\label{eq:y-corrector}
			y_{k+1}=\arg\min_y \mathcal L_{\rho_k}(x_{k+1},y,\tau\lambda_k).
		\end{equation}

		\STATE Update $\lambda_{k+1}$ by
		\begin{equation}\label{eq:lambda-full}
			\lambda_{k+1}=\tau\lambda_k-\rho_k(Ax_{k+1}-y_{k+1}).
		\end{equation}
		
		\STATE Draw a fresh common batch \(\mathcal S_{k+1}\) and compute \(v_{k+1}\) by \eqref{eq:storm}.
		
		\ENDFOR
		
		\STATE Draw \(J\) uniformly and independently from
		\(\{\lceil K/2\rceil,\ldots,K\}\).
		
		\STATE \textbf{return} \((x_{J+1},y_{J+1},\lambda_{J+1})\).
	\end{algorithmic}
\end{algorithm}

Some comments about Algorithm \ref{alg1} are in order. 
First, the updates in \eqref{eq:y-predictor} and \eqref{eq:y-corrector} are two proximal operators:
\[
\widehat y_{k+1}=\operatorname{prox}_{g/\rho_k}\left(Ax_k-\frac{\tau\lambda_k}{\rho_k}\right), 
\quad 
y_{k+1}=\operatorname{prox}_{g/\rho_k}\left(Ax_{k+1}-\frac{\tau\lambda_k}{\rho_k}\right).
\]
These correspond to the forward and backward sweeps of the nonsmooth block, respectively. 
Structurally, the two-stage \(y\)-update resembles the symmetric Gauss-Seidel iteration proposed in \cite{li2016schur,li2019blocksgs}, where the sweep decomposes a large-block subproblem into smaller exactly solvable pieces while preserving equivalence to the standard ADMM in the convex setting. Thus, the algorithmic roles are fundamentally different.

Second, \eqref{eq:direction} preserves the exact augmented geometry in each $x$-step and avoids introducing additional variance into the coupling. The estimator $\mathcal{G}_k$ is constructed by handling the stochastic part $F(x)$ and the deterministic part $-\langle \tau\lambda_k, Ax_k-\widehat y_{k+1} \rangle+\frac{\rho_k}{2}\|Ax_k- \widehat y_{k+1}\|^2$ separately. 
Specifically, ${\rm grad}~F(x)$ is estimated using recursive momentum \cite{cutkosky2019momentum,han2021riemannian}, while $-\langle \tau\lambda_k, Ax_k-\widehat y_{k+1} \rangle+\frac{\rho_k}{2}\|Ax_k- \widehat y_{k+1}\|^2$ is computed exactly as $\mathcal{P}_{T_{x_k}\mathcal{M}}\big(\rho_k A^\top( Ax_k - \widehat y_{k+1} - \tau\lambda_k /\rho_k)\big)$. 
In contrast, $\mathcal{G}_k$ in \cite{deng2024oracle,shi2025momentum} is applied directly to the entire augmented Lagrangian function $\mathcal L_{\rho_k}$, treating $\rho_k$ as part of the variables. 

Third, the damped dual update \eqref{eq:lambda-full} has recently been used to bound the dual sequence in the convergence analysis of nonconvex operator-splitting methods \cite{hajinezhad2019perturbed,kong2024global}. This approach can be interpreted as a regularized version of the classical dual ascent step \cite{sun2024dual}. Our analysis shows that the damping strategy introduces contraction, resulting in nonsingular envelope and tracking bounds, as detailed in Lemma \ref{lem:tracking}.  

Fourth, the penalty parameter $\rho_k = c_\rho k^{1/3}$ is designed to increase monotonically, a strategy that is well supported in theory \cite{he2000alternating,lin2011linearized,dolgopolik2022exact}. In practice, its growth is moderate. For example, after $K=1000$ iterations, $\rho_K = 10\rho_1$, thus avoiding the numerical instability of more aggressive schedules.

\section{Convergence Analysis}\label{sec4}
This section analyzes the convergence properties of Algorithm 1. 

\subsection{Assumptions}
We list the assumptions necessary for the convergence analysis.
\begin{assumption}\label{ass1}
	(i) The set \(\mathcal M\subset\mathbb R^n\) is a compact smooth embedded submanifold. 
	
	(ii) The function \(F\) is continuously differentiable on an open set containing \({\rm conv}(\mathcal M)\), and its Euclidean gradient is \(L_F\)-Lipschitz there. 
	
	(iii) The function \(g:\mathbb R^m\to\mathbb R\) is finite, closed, convex, globally \(L_g\)-Lipschitz, and its ordinary Euclidean proximal map is available.
\end{assumption}

\begin{assumption}\label{ass2}
	For all \(x\in\mathcal M\), the stochastic Riemannian gradient is unbiased and has uniformly bounded variance:
	\[
	\mathbb E_\xi[\operatorname{grad}f(x,\xi)]=\operatorname{grad}F(x),\quad
	\mathbb E_\xi\|\operatorname{grad}f(x,\xi)-\operatorname{grad}F(x)\|^2\leq\sigma^2.
	\]
\end{assumption}

\begin{assumption}\label{ass3}
	There is a constant \(\widetilde L>0\) such that, for all \(x\in\mathcal M\) and \(u\in T_x\mathcal M\),
	\[
	\mathbb E_\xi\left\|\operatorname{grad}f(\mathcal R_x(u),\xi)-\mathcal T_x^{\mathcal R_x(u)}\operatorname{grad}f(x,\xi)\right\|^2
	\leq\widetilde L^2\|u\|^2.
	\]
	The transport is the parallel transport fixed in Section \ref{sec2}.
\end{assumption}

The global \(L_g\)-Lipschitz property in Assumption \ref{ass1} is equivalent to \(\operatorname{dom}g^*\subseteq \{z:\|z\|\leq L_g\}\). This covers norms and standard finite-valued sparsity regularizers, including \(\ell_1\), group, and nuclear norms in finite dimensions. General indicator functions are outside the theorem because their normal-cone multipliers need not be bounded. 
Assumption \ref{ass3} is the standard common-sample mean-squared smoothness condition for recursive Riemannian estimators \cite{han2021riemannian,deng2024oracle}. It is a second-moment condition and does not impose an almost-sure stochastic gradient bound. A convenient sufficient condition is for the ambient sample gradient to have a random Lipschitz modulus $L(\xi)$ with $\mathbb E[L(\xi)^2]<\infty$ and a finite second moment at one reference point. On a compact embedded manifold, the smoothness of the tangent projection, the retraction, and parallel transport then yields Assumption \ref{ass3} with a finite $\widetilde L$. This covers the sphere and the Stiefel manifold with the usual smooth sample losses and bounded or finite-second-moment data. It can fail for nonsmooth sample losses or heavy-tailed oracles lacking the required second moments.

\subsection{Theoretical Results}
The proof has four ledgers. First, the predictor and corrector dual images contract toward the dual maximizer of a Moreau envelope. Second, the recursive stochastic estimator admits a cumulative error bound through self-normalization. Third, the moving target caused by \(x_k\) and \(\rho_k\) is controlled and square-summable relative penalty increments.  Finally, a pathwise Abel argument accommodates the random inverse stepsizes.

\begin{lemma}\label{lem:envelope}
	Suppose that Assumption \ref{ass1} holds. 
	For any \(\rho>0\), \(\tau\in(0,1)\) and \(x,x'\in\mathcal M\), define
	\[
	H_\rho(x)=\max_{z\in\mathbb R^m}\left\{\langle z,Ax\rangle-g^*(z)-\frac{1-\tau}{2\rho}\|z\|^2\right\},
	\]
	and its unique maximizer \(w_\rho(x)\). Then it holds that
	\begin{align}
		\operatorname{grad}H_\rho(x)&=\mathcal P_{T_x\mathcal M}(A^\top w_\rho(x)),\label{eq:envelope-gradient}\\
		\|w_\rho(x)\|&\le L_g,\label{eq:w-bound}\\
		\|w_\rho(x)-w_\rho(x')\|&\le\frac{\rho\|A\|}{1-\tau}\|x-x'\|,\label{eq:w-x-lipschitz}\\
		g(Ax)-\frac{(1-\tau) L_g^2}{2\rho}&\le H_\rho(x)\le g(Ax).\label{eq:envelope-bias}
	\end{align}
	Furthermore, if \(\rho'\ge\rho\), then
	\begin{align}
		\|w_{\rho'}(x)-w_\rho(x)\|&\leq L_g\left(\frac{\rho'}{\rho}-1\right),\label{eq:w-rho-sensitivity}\\
		0\leq H_{\rho'}(x)-H_\rho(x)&\leq\frac{(1-\tau) L_g^2}{2}\left(\frac1\rho-\frac1{\rho'}\right).\label{eq:envelope-drift}
	\end{align}
\end{lemma}
\begin{proof}
	The maximization problem in $H_\rho(x)$ is strongly concave on \(\operatorname{dom}g^*\), so the maximizer \(w_\rho(x)\) is unique. 
	Its optimality condition is
	\begin{equation}\label{eq:w-optimality}
		Ax-\frac{1-\tau}{\rho}w_\rho(x)\in\partial g^*(w_\rho(x)).
	\end{equation}
	Since \(H_\rho(x)=\max_z h(x,z)\) with \(h(x,z)=\langle z,Ax\rangle-g^*(z)-\frac{1-\tau}{2\rho}\|z\|^2\) and its maximum is attained at the unique point \(w_\rho(x)\), Danskin's theorem implies that 
	\[
	\nabla H_\rho(x)=\nabla_x h(x,w_\rho(x))=A^\top w_\rho(x).
	\]
	Thus its Riemannian gradient is \eqref{eq:envelope-gradient}. 
	The function \(g\) is \(L_g\)-Lipschitz, so its convex conjugate \(g^*\) has domain contained in the closed ball of radius \(L_g\):	\(\operatorname{dom}g^*\subseteq\{z:\|z\|\le L_g\}\). Since the maximization in the definition of \(H_\rho\) is over \(z\in\operatorname{dom}g^*\), we have \(w_\rho(x)\in\operatorname{dom}g^*\), and therefore, \(\|w_\rho(x)\|\le L_g\), which is \eqref{eq:w-bound}. 
	
	By applying the monotonicity of \(\partial g^*\) to \eqref{eq:w-optimality} at \(x\) and \(x'\), we obtain that
	\[
	\begin{aligned}
		0\le&
		\left\langle w_\rho(x)-w_\rho(x'),
		A(x-x')-\frac{1-\tau}{\rho}(w_\rho(x)-w_\rho(x'))\right\rangle\\
		=& \left\langle w_\rho(x)-w_\rho(x'),A(x-x')\right\rangle
		-\frac{1-\tau}{\rho}\|w_\rho(x)-w_\rho(x')\|^2\\
		\le& \|A\|\|w_\rho(x)-w_\rho(x')\|\|x-x'\|
		-\frac{1-\tau}{\rho}\|w_\rho(x)-w_\rho(x')\|^2.
	\end{aligned}
	\]
	This gives \eqref{eq:w-x-lipschitz}.
	Note that the Fenchel duality theorem yields 
	\[
	H_\rho(x)=\min_y
	\left\{g(y)+\frac{\rho}{2(1-\tau)}\|y-Ax\|^2\right\}.
	\]
	Taking \(y=Ax\) proves the upper bound in \eqref{eq:envelope-bias}. Let \(t=\|y-Ax\|\). The Lipschitz continuity of \(g\) implies \(g(y)\ge g(Ax)-L_g t\). Consequently,
	\[
	g(y)+\frac{\rho}{2(1-\tau)}t^2
	\ge g(Ax)-L_gt+\frac{\rho}{2(1-\tau)}t^2
	\ge g(Ax)-\frac{(1-\tau) L_g^2}{2\rho},
	\]
	where the last inequality follows from the fact that $-L_gt+\frac{\rho}{2(1-\tau)}t^2$ is a quadratic in \(t\) and attains its minimum at \(t_{\min}=\frac{(1-\tau)L_g}{\rho}\), with the minimal value being \(-\frac{(1-\tau) L_g^2}{2\rho}\).
	
	Fix \(x\) and let \(\rho'\ge\rho>0\). The monotonicity of the two instances of \eqref{eq:w-optimality} yields
	\[\begin{aligned}
		0\leq&\Bigl\langle\Bigl(Ax-\frac{1-\tau}{\rho'}w_{\rho'}(x)\Bigr)-\Bigl(Ax-\frac{1-\tau}{\rho}w_\rho(x)\Bigr),\,w_{\rho'}(x)-w_\rho(x)\Bigr\rangle\\
		=&\Bigl\langle\frac{1-\tau}{\rho}w_\rho(x)-\frac{1-\tau}{\rho'}w_{\rho'}(x),\,w_{\rho'}(x)-w_\rho(x)\Bigr\rangle\\
		=&\frac{1-\tau}{\rho}\langle w_\rho(x),w_{\rho'}(x)-w_\rho(x)\rangle
		-\frac{1-\tau}{\rho'}\bigl(\langle w_\rho(x),w_{\rho'}(x)-w_\rho(x)\rangle+\|w_{\rho'}(x)-w_\rho(x)\|^2\bigr).
	\end{aligned}\]
	Dividing by \(1-\tau>0\) and rearranging, we have
	\[\begin{aligned}
		\frac1{\rho'}\|w_{\rho'}(x)-w_\rho(x)\|^2
		&\leq \left(\frac1\rho-\frac1{\rho'}\right)\langle w_{\rho'}(x)-w_\rho(x),w_\rho(x)\rangle\\
		&\leq \left(\frac1\rho-\frac1{\rho'}\right)\|w_{\rho'}(x)-w_\rho(x)\|\|w_\rho(x)\|\\
		&\overset{\eqref{eq:w-bound}}{\leq} L_g\left(\frac1\rho-\frac1{\rho'}\right)\|w_{\rho'}(x)-w_\rho(x)\|,
	\end{aligned}\]
	which proves \eqref{eq:w-rho-sensitivity}. Finally, the defining objective of \(H_{\rho'}\) is pointwise no smaller than that of \(H_\rho\), establishing the lower bound in \eqref{eq:envelope-drift}. 
	For the upper bound, observe that
	\[
	H_{\rho'}(x)=\langle w_{\rho'}(x),Ax\rangle-g^*(w_{\rho'}(x))-\frac{1-\tau}{2\rho'}\|w_{\rho'}(x)\|^2,
	\]
	\[
	H_\rho(x)\ge\langle w_{\rho'}(x),Ax\rangle-g^*(w_{\rho'}(x))-\frac{1-\tau}{2\rho}\|w_{\rho'}(x)\|^2.
	\]
	Subtracting these two inequalities yields
	\[
	H_{\rho'}(x)-H_\rho(x)\leq\frac{1-\tau}{2}\Bigl(\frac1\rho-\frac1{\rho'}\Bigr)\|w_{\rho'}(x)\|^2
	\overset{\eqref{eq:w-bound}}{\leq}\frac{(1-\tau) L_g^2}{2}\Bigl(\frac1\rho-\frac1{\rho'}\Bigr),
	\]
	which proves \eqref{eq:envelope-drift}.
\end{proof}

Controlling the dual sequence is often the central obstacle in convergence analysis for nonconvex primal-dual methods. In this study, we introduce a novel analytical approach that treats the multiplier update as a contraction mapping toward the dual maximizer of a Moreau envelope. 

\begin{lemma}\label{lem:tracking}
	Define the $\tau$-contraction mapping
	\[
	T_{\rho,x}(z):=\operatorname{prox}_{\rho g^*}(\tau z+\rho Ax).
	\]
	Then \(w_\rho(x)\) is its unique fixed point. 
	Suppose that Assumption \ref{ass1} holds and consider the sequence $\{x_k, y_k, \lambda_k\}_1^K$ generated by Algorithm \ref{alg1}. Let $\widehat\lambda_{k+1}:=\tau\lambda_k-\rho_k(Ax_k-\widehat y_{k+1})$. Then Algorithm \ref{alg1} satisfies,
	\[
	-\lambda_{k+1}=T_{\rho_k,x_{k+1}}(-\lambda_k),\quad -\widehat\lambda_{k+1}=T_{\rho_k,x_k}(-\lambda_k).
	\]
	Let \(\varrho\) be the retraction constant in Lemma \ref{proposi:ret-lips}. Moreover, it follows that
	\[
	\sum_{k=1}^K\mathbb E\|\lambda_k+w_{\rho_k}(x_k)\|^2\leq 
	\frac{4c_\rho^2 c_\eta^2\tau^2\varrho^2\|A\|^2(1+\tau^2) }{(1-\tau)^2(1-\tau^2)^2}\sum_{k=1}^K\mathbb E\|\mathcal G_k\|^2
	+\frac{8L_g^2}{1-\tau^2}+\frac{2\pi^2L_g^2(1+\tau^2)}{27(1-\tau^2)^2}.
	\]
\end{lemma}
\begin{proof}
	Since the fixed-point optimality condition is exactly $Ax-\frac{1-\tau}{\rho}w_\rho(x)\in\partial g^*(w_\rho(x))$, \(w_\rho(x)\) is the unique fixed point. 
	The Moreau decomposition yields  \[y_{k+1}=\operatorname{prox}_{g/\rho_k}\left(Ax_{k+1}-\frac{\tau\lambda_k}{\rho_k}\right)=Ax_{k+1}-\frac{\tau\lambda_k}{\rho_k}-\frac{1}{\rho_k}\operatorname{prox}_{\rho_k g^*}\left(\rho_k Ax_{k+1}-\tau\lambda_k\right).\]
	With \eqref{eq:lambda-full}, we have
	\[
	-\lambda_{k+1}=\operatorname{prox}_{\rho_k g^*}(\rho_kAx_{k+1}-\tau \lambda_k)=T_{\rho_k,x_{k+1}}(-\lambda_k).
	\]
	Similarly, 
	\[
	\widehat y_{k+1}=\operatorname{prox}_{g/\rho_k}\left(Ax_k-\frac{\tau\lambda_k}{\rho_k}\right)=Ax_k-\frac{\tau\lambda_k}{\rho_k}-\frac{1}{\rho_k}\operatorname{prox}_{\rho_k g^*}\left(\rho_k Ax_k-\tau\lambda_k\right).
	\]
	With $\widehat\lambda_{k+1}=\tau\lambda_k-\rho_k(Ax_k-\widehat y_{k+1})$, we obtain
	\[
	-\widehat\lambda_{k+1}=\operatorname{prox}_{\rho_k g^*}(\rho_kAx_k-\tau \lambda_k)=T_{\rho_k,x_k}(-\lambda_k).
	\]
	Contraction at the corrected base point, followed by \eqref{eq:w-x-lipschitz} and \eqref{eq:w-rho-sensitivity}, yields
	\[
	\begin{aligned}
		\|\lambda_{k+1}+w_{\rho_{k+1}}(x_{k+1})\|
		\leq&\|\lambda_{k+1}+w_{\rho_k}(x_k)\|+\|w_{\rho_k}(x_{k+1})-w_{\rho_k}(x_k)\|+\|w_{\rho_{k+1}}(x_{k+1})-w_{\rho_k}(x_{k+1})\|\\
		\leq&\tau \|\lambda_k+w_{\rho_k}(x_k)\|+\frac{\tau\rho_k\|A\|}{1-\tau}\|x_{k+1}-x_k\|+L_g\big(\frac{\rho_{k+1}}{\rho_k}-1\big)\\
		\leq&\tau \|\lambda_k+w_{\rho_k}(x_k)\|+\frac{\tau \varrho\|A\|}{1-\tau}\rho_k\eta_k \|\mathcal G_k\|+L_g\big(\frac{\rho_{k+1}}{\rho_k}-1\big),
	\end{aligned}
	\]
	For the three terms of the RHS, applying
	\[
	(a+(b+c))^2\leq(1+\nu)a^2+(1+\nu^{-1})(b+c)^2\leq (1+\nu)a^2+2(1+\nu^{-1}) (b^2+c^2),
	\]
	with \(\nu=(1-\tau^2)/(2\tau^2)\) to $a=\tau \|\lambda_k+w_{\rho_k}(x_k)\|$, $b=\frac{\tau \varrho\|A\|}{1-\tau}\rho_k\eta_k\|\mathcal G_k\|$ and $c=L_g\big(\frac{\rho_{k+1}}{\rho_k}-1\big)$, we obtain
	\begin{align}
		\|\lambda_{k+1}+w_{\rho_{k+1}}(x_{k+1})\|^2\leq&\frac{1+\tau^2}{2} \|\lambda_k+w_{\rho_k}(x_k)\|^2+\frac{2\tau^2\varrho^2\|A\|^2(1+\tau^2)}{(1-\tau)^2(1-\tau^2)}\rho_k^2\eta_k^2\|\mathcal G_k\|^2\nonumber\\
		&+ \frac{2L_g^2(1+\tau^2)}{1-\tau^2} \big(\frac{\rho_{k+1}}{\rho_k}-1\big)^2.\label{eq:q-recursion}
	\end{align}
	The Bernoulli inequality gives
	\[
	0\le \frac{\rho_{k+1}}{\rho_k}-1 =\big(1+\frac{1}{k}\big)^{1/3}-1\le\frac{1}{3k}.
	\]
	Then we have 
	\[
	\sum_{k=1}^{\infty}(\frac{\rho_{k+1}}{\rho_k}-1)^2\le\sum_{k=1}^{\infty}\frac{1}{9k^2}\le\frac{\pi^2}{54}.
	\]
	Moreau's identity gives \(-\lambda_1=\operatorname{prox}_{g^*}(0)\), so the initialization gives \(-\lambda_1\in\operatorname{dom}g^*\). 
	Therefore, both \(\lambda_1\) and \(w_{\rho_1}(x_1)\) belong to the ball of radius \(L_g\), so \(\|\lambda_1+w_{\rho_1}(x_1)\|\le2L_g\). 
	Note that \eqref{eq:adaptive-schedules} implies that $\eta_k\le c_\eta k^{-1/3}$, and therefore, $\rho_k\eta_k\le c_\rho c_\eta$.
	Consequently, summing \eqref{eq:q-recursion} from \(k=1, \dots, K\) yields
	\[\begin{aligned}
		&\sum_{k=1}^K \|\lambda_k+w_{\rho_k}(x_k)\|^2
		\leq\|\lambda_1+w_{\rho_1}(x_1)\|^2+\sum_{k=1}^K \|\lambda_{k+1}+w_{\rho_{k+1}}(x_{k+1})\|^2\\
		\leq& 4L_g^2+\sum_{k=1}^K \frac{1+\tau^2}{2} \|\lambda_k+w_{\rho_k}(x_k)\|^2
		+\sum_{k=1}^K \frac{2c_\rho^2 c_\eta^2\tau^2\varrho^2\|A\|^2(1+\tau^2)}{(1-\tau)^2(1-\tau^2)}  \|\mathcal G_{k-1}\|^2+\frac{\pi^2L_g^2 (1+\tau^2)}{27(1-\tau^2)}.
	\end{aligned}\]
	Taking expectations gives the result.
\end{proof}

The following lemmas provide a recursive bound on the error $\|v_k-\operatorname{grad}F(x_k)\|$. For notation simplicity, we denote $D_k=D_{k-1}+1+\frac{\|\mathcal G_k\|^2}{\alpha_{k+1}}$ for $k\geq 1$, with $D_0=0$. Then the gradient stepsize $\eta_k=c_\eta D_k^{-1/3}$.
\begin{lemma}\cite[Lemma 1]{han2021riemannian}\label{lem:error-rec}
	Suppose that Assumptions \ref{ass2} and \ref{ass3} hold, and consider Algorithm \ref{alg1}. Then the expected estimation error can be bounded as follows:
	\[
	\mathbb{E}\|v_{k+1}-\operatorname{grad}F(x_{k+1})\|^2\leq 
	(1-\alpha_{k+1})^2\mathbb{E}\|v_k-\operatorname{grad}F(x_k)\|^2
	+2\eta_k^2\tilde{L}^2\mathbb{E}\|\mathcal{G}_k\|^2+2\alpha_{k+1}^2\sigma^2.
	\]
\end{lemma}

\begin{lemma}\label{lem:error-sum}
	Suppose that Assumptions \ref{ass2} and \ref{ass3} hold, and consider the sequence $\{x_k, y_k, \lambda_k\}_1^K$ generated by Algorithm \ref{alg1}. 
	Set $\frac{2}{3}<c_\alpha\leq 1$. Then it holds that
	\[
	\sum_{k=1}^K \mathbb E\|v_k-\operatorname{grad}F(x_k)\|^2\leq
	\frac{18c_\alpha c_\eta^2 \widetilde L^2}{3c_\alpha-2}\mathbb E[D_K^{1/3}]
	+\frac{3 \sigma^2(1+6c_\alpha^2)}{3c_\alpha-2} K^{1/3}.
	\]
\end{lemma}
\begin{proof}
	Combining Lemma \ref{lem:error-rec} with \((1-\alpha_{k+1})^2\le1-\alpha_{k+1}\) yields
	\begin{align}
		\mathbb E\|v_k-\operatorname{grad}F(x_k)\|^2
		\leq&\frac{\mathbb E\|v_k-\operatorname{grad}F(x_k)\|^2-\mathbb E\|v_{k+1}-\operatorname{grad}F(x_{k+1})\|^2}{\alpha_{k+1}}\nonumber\\
		&+\frac{2\widetilde L^2\eta_k^2}{\alpha_{k+1}}\mathbb E\|\mathcal G_k\|^2+2\alpha_{k+1}\sigma^2.\label{eq:error}
	\end{align}
	For $k\geq 2$, the concavity of the function \(k^{2/3}\) gives
	\[
	\frac1{\alpha_{k+1}}-\frac1{a_k}
	=\frac{k^{2/3}-(k-1)^{2/3}}{c_\alpha}
	\le\frac{2}{3c_\alpha}.
	\]
	Then, for \(k\geq 1\), Abel's summation and Assumption \ref{ass2} lead to the inequality
	\begin{align}
		&\sum_{k=1}^K\frac{\mathbb E\|v_k-\operatorname{grad}F(x_k)\|^2-\mathbb E\|v_{k+1}-\operatorname{grad}F(x_{k+1})\|^2}{\alpha_{k+1}}\nonumber\\
		=&\frac{\mathbb E\|v_1-\operatorname{grad}F(x_1)\|^2}{a_2}
		+\sum_{k=2}^K\left(\frac1{\alpha_{k+1}}-\frac1{a_k}\right)\mathbb E\|v_k-\operatorname{grad}F(x_k)\|^2\nonumber\\
		&-\frac{\mathbb E\|v_{K+1}-\operatorname{grad}F(x_{K+1})\|^2}{a_{K+1}}\nonumber\\
		\leq&\frac{\sigma^2}{c_\alpha}+\frac{2}{3c_\alpha}\sum_{k=2}^K \mathbb E\|v_k-\operatorname{grad}F(x_k)\|^2.\label{eq:varepsilon-sum}
	\end{align}
	Since $D_k-D_{k-1}=1+\frac{\|\mathcal G_k\|^2}{\alpha_{k+1}}\geq \frac{\|\mathcal G_k\|^2}{\alpha_{k+1}}$ and $\eta_k=c_\eta D_k^{-1/3}$, we have that
	\begin{align}
		\sum_{k=1}^K\frac{\eta_k^2\|\mathcal G_k\|^2}{\alpha_{k+1}}
		&=\sum_{k=1}^K \frac{c_\eta^2\|\mathcal G_k\|^2}{D_k^{2/3} \alpha_{k+1}}
		\leq c_\eta^2\sum_{k=1}^K\frac{D_k-D_{k-1}}{D_k^{2/3}}\nonumber\\
		&\leq c_\eta^2\sum_{k=1}^K\int_{D_{k-1}}^{D_k}t^{-2/3}\,dt
		=c_\eta^2\sum_{k=1}^K 3(D_k^{1/3}-D_{k-1}^{1/3})
		\leq 3c_\eta^2 D_K^{1/3}.\label{eq:self-normalization}
	\end{align}
	It also holds that \(\sum_{k=1}^K\alpha_{k+1}\le 3c_\alpha K^{1/3}\). 
	Finally, summing \eqref{eq:error}, substituting \eqref{eq:varepsilon-sum}, and applying \eqref{eq:self-normalization} pathwise before taking expectation yields
	\[
	\sum_{k=1}^K\mathbb E\|v_k-\operatorname{grad}F(x_k)\|^2\le
	\frac{\sigma^2}{c_\alpha}+\frac{2}{3c_\alpha}\sum_{k=2}^K \mathbb E\|v_k-\operatorname{grad}F(x_k)\|^2
	+6c_\eta^2\widetilde L^2 \mathbb E [D_K^{1/3}]
	+6c_\alpha \sigma^2  K^{1/3}.
	\]
	Rearranging the above inequality results in the claim.
\end{proof}

Define the merit function $\Phi_\rho(x):=F(x)+H_\rho(x)$. It has the following uniform retraction descent property.
\begin{lemma}\label{lem:envelope-descent} 
	Suppose that Assumption \ref{ass1} holds and consider Algorithm \ref{alg1}. Let $\varrho$ and $\beta$ be the retraction constants in Lemma \ref{proposi:ret-lips}, and let \(L_F^R\) be the retraction-smoothness constant for \(F\) in Lemma \ref{lem:Lpb}. 
	For all $k\geq 1$, and any \(x\in\mathcal M\), \(u\in T_x\mathcal M\), it holds that
	\[
	\Phi_{\rho_k}(\mathcal R_x(u))
	\leq \Phi_{\rho_k}(x)+\langle\operatorname{grad}\Phi_{\rho_k}(x),u\rangle
	+\frac{M_\Phi\rho_k}{2}\|u\|^2,
	\]
	where $M_\Phi:=\frac{L_F^R+2\beta\|A\|L_g}{c_\rho}+\frac{\varrho^2\|A\|^2}{1-\tau}$.
\end{lemma}
\begin{proof}
	Note that $H_\rho(x)=M_g^{(1-\tau)/\rho}(Ax)$ and $w_\rho(x)=\nabla M_g^{(1-\tau)/\rho}(Ax)$. From Lemma \ref{lem:prox} and \eqref{eq:w-bound}, the Euclidean gradient of the Moreau envelope \(M_g^{(1-\tau)/\rho}\) is \(\frac{\rho}{1-\tau}\)-Lipschitz and has norm at most \(L_g\). Furthermore, we have
	\[\nabla H_\rho(x)=A^\top \nabla M_g^{(1-\tau)/\rho}(Ax)=A^\top w_\rho(x),\]
	which is \(\frac{\rho\|A\|^2}{1-\tau}\)-Lipschitz. Since \(u\in T_x\mathcal M\), we have $\langle A^\top w,u\rangle=\langle\mathcal P_{T_x\mathcal M}(A^\top w),u\rangle$. Consequently, the descent lemma and Lemma \ref{proposi:ret-lips} imply that
	\[
	\begin{aligned}
		H_\rho(\mathcal R_x(u))
		&\leq H_\rho(x)+\langle A^\top w_\rho(x),\mathcal R_x(u)-x\rangle+\frac{\rho_k\|A\|^2}{2(1-\tau)}\|\mathcal R_x(u)-x\|^2\\
		&= H_\rho(x)+\langle A^\top w_\rho(x),u\rangle+\langle A^\top w_\rho(x),\mathcal R_x(u)-x-u\rangle+\frac{\rho_k\|A\|^2}{2(1-\tau)}\|\mathcal R_x(u)-x\|^2\\
		&\leq H_\rho(x)+\langle A^\top w_\rho(x),u\rangle+\|A\|\|w_\rho(x)\|\|\mathcal R_x(u)-x-u\|+\frac{\rho_k\|A\|^2}{2(1-\tau)}\|\mathcal R_x(u)-x\|^2\\
		&\leq H_\rho(x)+\langle\mathcal P_{T_x\mathcal M}(A^\top w_\rho(x)),u\rangle
		+\beta\|A\|L_g\|u\|^2+\frac{\varrho^2\rho_k\|A\|^2}{2(1-\tau)}\|u\|^2.
	\end{aligned}
	\]
	On the other hand, the retraction descent bound for \(F\) is
	\[F(R_x(u))\leq F(x)+\langle {\rm grad}\, F(x),u \rangle +\frac{L_F^R}{2}\|u\|^2.\]
	By the definition of $\Phi_\rho(x)$, we have
	\[
	\Phi_{\rho_k}(\mathcal R_x(u))\leq \Phi_{\rho_k}(x)+\langle\operatorname{grad}\Phi_{\rho_k}(x),u\rangle
	+\left(\frac{L_F^R+2\beta\|A\|L_g}{\rho_k}+\frac{\varrho^2\|A\|^2}{1-\tau}\right)\frac{\rho_k}{2}\|u\|^2.
	\]
	The proof is completed by $\rho_k=c_\rho k^{1/3}\geq c_\rho$ for $k\geq 1$. 
\end{proof}

The following lemma further presents the one-step coupled descent of the moving target $\Phi_{\rho}(x)$ caused by $x_k$ and $\rho_k$.
\begin{lemma}\label{lem:coupled-descent}
	Suppose that Assumption \ref{ass1} holds and consider Algorithm \ref{alg1}. 
	Set $c_\rho>0$ and $c_\eta\leq \frac{1}{2M_\Phi c_\rho}$, where $M_\Phi$ is a constant defined in Lemma \ref{lem:envelope-descent}. Then, for all $k\geq 1$, we have
	\[\begin{aligned}
		\Phi_{\rho_{k+1}}(x_{k+1})-\Phi_{\rho_k}(x_k)
		\leq&-\frac14\eta_k\|\mathcal G_k\|^2
		+\eta_k\|v_k-\operatorname{grad}F(x_k)\|^2\\
		&+\eta_k\tau^2\|A\|^2\|\lambda_k+w_{\rho_k}(x_k)\|^2+\frac{(1-\tau) L_g^2}{2}\left(\frac1{\rho_k}-\frac1{\rho_{k+1}}\right).
	\end{aligned}\]
\end{lemma}
\begin{proof}
	From \eqref{eq:direction}, \eqref{eq:envelope-gradient} and $\widehat\lambda_{k+1}=\tau\lambda_k-\rho_k(Ax_k-\widehat y_{k+1})$, we have
	\[
	\operatorname{grad}\Phi_{\rho_k}(x_k)
	=\mathcal G_k-\big(v_k-\operatorname{grad}F(x_k)\big)+\mathcal P_{T_{x_k}\mathcal M}\big(A^\top(w_{\rho_k}(x_k)+\widehat\lambda_{k+1})\big).
	\]
	Using the fixed-point identity and the $\tau$-contraction of $T_{\rho,x}$, it follows Lemma \ref{lem:tracking} that 
	\[\begin{aligned}
		\|\mathcal P_{T_{x_k}\mathcal M}\big(A^\top(w_{\rho_k}(x_k)+\widehat\lambda_{k+1})\big)\|
		\leq&\|A\|\|\widehat\lambda_{k+1}+w_{\rho_k}(x_k)\|\\
		=&\|A\|\|T_{\rho_k,x_k}(-\lambda_k)-T_{\rho_k,x_k}\big(w_{\rho_k}(x_k)\big)\|\\
		\leq&\tau\|A\|\|\lambda_k+w_{\rho_k}(x_k)\|.
	\end{aligned}\]
	Meanwhile,
	\[
	\langle v_k-\operatorname{grad}F(x_k),\mathcal G_k\rangle
	\le\|v_k-\operatorname{grad}F(x_k)\|^2+\frac14\|\mathcal G_k\|^2,
	\]
	and
	\[
	-\langle \mathcal P_{T_{x_k}\mathcal M}
	\big(A^\top(w_{\rho_k}(x_k)+\widehat\lambda_{k+1})\big),\mathcal G_k\rangle
	\le\tau^2\|A\|^2\|\lambda_k+w_{\rho_k}(x_k)\|^2+\frac14\|\mathcal G_k\|^2.
	\]
	By applying Lemma \ref{lem:envelope-descent} with \(u=-\eta_k\mathcal G_k\), we obtain
	\[\begin{aligned}
		&\Phi_{\rho_k}(x_{k+1})-\Phi_{\rho_k}(x_k)\\
		\leq& \langle \mathcal G_k-\big(v_k-\operatorname{grad}F(x_k)\big)+\mathcal P_{T_{x_k}\mathcal M}\big(A^\top(w_{\rho_k}(x_k)+\widehat\lambda_{k+1})\big), -\eta_k\mathcal G_k\rangle
		+\frac{M_\Phi\rho_k \eta_k^2}{2}\|\mathcal G_k\|^2\\
		\leq& -\eta_k\|\mathcal G_k\|^2+\eta_k\langle v_k-\operatorname{grad}F(x_k),\mathcal G_k\rangle-\eta_k\langle \mathcal P_{T_{x_k}\mathcal M}\big(A^\top(w_{\rho_k}(x_k)+\widehat\lambda_{k+1})\big),\mathcal G_k\rangle+\frac{\eta_k}{4}\|\mathcal G_k\|^2\\
		\leq&-\frac14\eta_k\|\mathcal G_k\|^2
		+\eta_k\|v_k-\operatorname{grad}F(x_k)\|^2+\eta_k\tau^2\|A\|^2\|\lambda_k+w_{\rho_k}(x_k)\|^2.
	\end{aligned}\]
	where \(M_\Phi\rho_k\eta_k^2\leq \frac{\eta_k}{2}\) follows from $M_\Phi c_\rho c_\eta\leq \frac{1}{2}$.  
	Finally, \eqref{eq:envelope-drift} at \(x_{k+1}\) gives
	\[\Phi_{\rho_{k+1}}(x_{k+1})-\Phi_{\rho_k}(x_{k+1})
	\leq \frac{(1-\tau) L_g^2}{2}\left(\frac1{\rho_k}-\frac1{\rho_{k+1}}\right). \]
	Adding the above two inequalities proves the result.
\end{proof}

Then, we prove the cumulative stochastic and tracking bounds for the coupled descent terms that appeared in Lemma \ref{lem:coupled-descent}. 
\begin{theorem}\label{the:sumG}
	Suppose that Assumptions \ref{ass1}--\ref{ass3} hold, and consider the sequence $\{x_k, y_k, \lambda_k\}_1^K$ generated by Algorithm \ref{alg1}.
	Let $\varrho$ be given in Lemma \ref{proposi:ret-lips}, and $M_\Phi$ be defined in Lemma \ref{lem:envelope-descent}. Set $\frac{2}{3}<c_\alpha\leq 1$, $c_\rho>0$ and $c_\eta\leq\min\left\{\frac{1}{2M_\Phi c_\rho},\frac{(1-\tau)(1-\tau^2)}{4 \varrho\tau^2\|A\|^2\sqrt{2(1+\tau^2)}}\right\}\frac{1}{c_\rho}$. Then we have
	\[
	\sum_{k=1}^K\mathbb E\|\mathcal G_k\|^2\leq C_{\mathcal G} K^{1/3},~
	\sum_{k=1}^K\mathbb E\|v_k-\operatorname{grad}F(x_k)\|^2\leq C_v K^{1/3},~
	\sum_{k=1}^K\mathbb E \|\lambda_k+w_{\rho_k}(x_k)\|^2\le C_\lambda K^{1/3},
	\]
	where 
	$\mathcal{C}_1:=\frac{\Phi_{\max}-\Phi_{\min}}{c_\eta}+\frac{18c_\alpha c_\eta^2 \widetilde L^2}{3c_\alpha-2}+\frac{(1-\tau) L_g^2}{2c_\rho c_\eta}$, 
	$\mathcal{C}_2:=\frac{3 \sigma^2(1+6c_\alpha^2)}{3c_\alpha-2}+\frac{8L_g^2\tau^2\|A\|^2}{1-\tau^2} +\frac{2L_g^2\pi^2\tau^2\|A\|^2(1+\tau^2)}{27(1-\tau^2)^2}$, 
	$C_{\mathcal G}:=16\big(\mathcal{C}_1+\mathcal{C}_2+4\mathcal{C}_1^{3/2}c_\alpha^{-1/2}\big)$, 
	$C_v:=\frac{18c_\alpha c_\eta^2 \widetilde L^2}{3c_\alpha-2}\left(1+\frac{C_{\mathcal G}}{c_\alpha}\right)^{1/3}+\frac{3 \sigma^2(1+6c_\alpha^2)}{3c_\alpha-2}$, 
	and $C_\lambda:= \frac{8L_g^2}{1-\tau^2}+\frac{2L_g^2\pi^2(1+\tau^2)}{27(1-\tau^2)^2}+\frac{4c_\rho^2 c_\eta^2\tau^2\varrho^2\|A\|^2(1+\tau^2) }{(1-\tau)^2(1-\tau^2)^2} C_{\mathcal G}$.
\end{theorem}
\begin{proof}
	First, the moving potential function is uniformly bounded. Let
	\[
	F_{\rm obj}^{\min}:=\min_{x\in\mathcal M}\{F(x)+g(Ax)\},\quad	
	F_{\rm obj}^{\max}:=\max_{x\in\mathcal M}\{F(x)+g(Ax)\}.
	\]
	It follows \eqref{eq:envelope-bias} that
	\[
	\Phi_{\min}:=F_{\rm obj}^{\min}-\frac{(1-\tau) L_g^2}{2c_\rho}
	\leq \Phi_{\rho}(x)\le F_{\rm obj}^{\max}=:\Phi_{\max}.
	\]
	Set $R_k=\Phi_{\rho_k}(x_k)-\Phi_{\min}\in[0,\Phi_{\max}-\Phi_{\min}]$. Abel's summation and $\eta_k=c_\eta D_k^{-1/3}$ yield the following result:
	\begin{align}
		\sum_{k=1}^K\frac{\Phi_{\rho_k}(x_k)-\Phi_{\rho_{k+1}}(x_{k+1})}{\eta_k}
		=&\frac{R_1}{\eta_1}+\sum_{k=2}^K\left(\frac1{\eta_k}-\frac1{\eta_{k-1}}\right)R_k-\frac{R_{K+1}}{\eta_K}\nonumber\\
		\leq &(\Phi_{\max}-\Phi_{\min})\left(\frac{1}{\eta_1}
		+\sum_{k=2}^K\left(\frac1{\eta_k}-\frac1{\eta_{k-1}}\right)\right)\nonumber\\
		\leq&\frac{\Phi_{\max}-\Phi_{\min}}{\eta_K}
		=\frac{\Phi_{\max}-\Phi_{\min}}{c_\eta}D_K^{1/3}.\label{eq:Phi}
	\end{align}
	Meanwhile, it holds that
	\begin{equation}\label{eq:rhok}
		\sum_{k=1}^K\frac{(1-\tau) L_g^2}{2\eta_k}\left(\frac1{\rho_k}-\frac1{\rho_{k+1}}\right)
		=\sum_{k=1}^K\frac{(1-\tau) L_g^2 D_k^{1/3}}{2c_\eta}\left(\frac1{\rho_1}-\frac1{\rho_{K+1}}\right)
		\leq\frac{(1-\tau) L_g^2}{2c_\rho c_\eta}D_K^{1/3}.
	\end{equation}
	Rearranging the result in Lemma \ref{lem:coupled-descent}, summing, and taking expectations yields
	\[\begin{aligned}
		\frac14\sum_{k=1}^K\mathbb E\|\mathcal G_k\|^2
		\leq &\mathbb E\sum_{k=1}^K\frac{\Phi_{\rho_k}(x_k)-\Phi_{\rho_{k+1}}(x_{k+1})}{\eta_k}
		+\sum_{k=1}^K\mathbb E\|v_k-\operatorname{grad}F(x_k)\|^2\\
		&+\tau^2\|A\|^2\sum_{k=1}^T\mathbb E \|\lambda_k+w_{\rho_k}(x_k)\|^2+\mathbb E\sum_{k=1}^K\frac{(1-\tau) L_g^2}{2\eta_k}\left(\frac1{\rho_k}-\frac1{\rho_{k+1}}\right)\\
		\leq& \frac{\Phi_{\max}-\Phi_{\min}}{c_\eta}\mathbb E[D_K^{1/3}]+\frac{18c_\alpha c_\eta^2 \widetilde L^2}{3c_\alpha-2}\mathbb E[D_K^{1/3}]+\frac{3 \sigma^2(1+6c_\alpha^2)}{3c_\alpha-2} K^{1/3}+\frac{8L_g^2\tau^2\|A\|^2}{1-\tau^2}\\
		&+\frac{2L_g^2\pi^2\tau^2\|A\|^2(1+\tau^2)}{27(1-\tau^2)^2} 
		+\frac{4c_\rho^2 c_\eta^2\varrho^2\tau^4\|A\|^4(1+\tau^2) }{(1-\tau)^2(1-\tau^2)^2}\sum_{k=1}^K\mathbb E\|\mathcal G_k\|^2
		+\frac{(1-\tau) L_g^2}{2c_\rho c_\eta}\mathbb E[D_K^{1/3}].
	\end{aligned}\]
	where the last inequality follows by \eqref{eq:Phi}, \eqref{eq:rhok}, Lemmas \ref{lem:tracking} and \ref{lem:error-sum}. 
	The parameter setting means that \(\frac{32c_\rho^2 c_\eta^2\varrho^2\tau^4\|A\|^4(1+\tau^2) }{(1-\tau)^2(1-\tau^2)^2}\leq \frac{1}{8}\). Then we have
	\begin{equation}\label{eq:D-vs-S}
		\sum_{k=1}^K\mathbb E\|\mathcal G_k\|^2\leq
		8\mathcal{C}_1\mathbb E[D_K^{1/3}]+8\mathcal{C}_2K^{1/3}.
	\end{equation}
	By $D_k=D_{k-1}+1+\frac{\|\mathcal G_k\|^2}{\alpha_{k+1}}$ and $\eta_k=c_\eta D_k^{-1/3}$ with $D_0=0$, it follows \eqref{eq:adaptive-schedules} that
	\begin{equation}\label{eq:ESK1/3}
		\mathbb E [D_K]=D_0+K+\sum_{k=1}^K \frac{\mathbb E\|\mathcal G_k\|^2}{\alpha_{k+1}}
		= K+\frac{K^{2/3}}{c_\alpha}\sum_{k=1}^K\mathbb E\|\mathcal G_k\|^2.
	\end{equation}
	Jensen's inequality and the subadditivity of the cube root imply that
	\[	
	\mathbb E[D_K^{1/3}] 
	\leq \left(K+\frac{K^{2/3}}{c_\alpha}\sum_{k=1}^K\mathbb E\|\mathcal G_k\|^2\right)^{1/3}
	\leq K^{1/3}+\frac{K^{2/9}}{c_\alpha^{1/3}}\left(\sum_{k=1}^K\mathbb E\|\mathcal G_k\|^2\right)^{1/3}.
	\]
	Thus, \eqref{eq:D-vs-S} becomes
	\begin{equation}\label{eq:scalar-closure}
		\frac{\sum_{k=1}^K\mathbb E\|\mathcal G_k\|^2}{K^{1/3}}\leq
		8(\mathcal{C}_1+\mathcal{C}_2)+\frac{8\mathcal{C}_1}{c_\alpha^{1/3}}\left(\frac{\sum_{k=1}^K\mathbb E\|\mathcal G_k\|^2}{K^{1/3}}\right)^{1/3}.
	\end{equation}
	If \(\frac{\sum_{k=1}^K\mathbb E\|\mathcal G_k\|^2}{K^{1/3}}\leq 16(\mathcal{C}_1+\mathcal{C}_2)\), it implies that
	\[
	\sum_{k=1}^K\mathbb E\|\mathcal G_k\|^2\leq 16(\mathcal{C}_1+\mathcal{C}_2) K^{1/3}.
	\]
	Otherwise, \(8(\mathcal{C}_1+\mathcal{C}_2)\leq\frac{\sum_{k=1}^K\mathbb E\|\mathcal G_k\|^2}{2K^{1/3}}\), then \eqref{eq:scalar-closure} gives
	\[\frac{\sum_{k=1}^K\mathbb E\|\mathcal G_k\|^2}{2K^{1/3}}\leq \frac{8\mathcal{C}_1}{c_\alpha^{1/3}}\left(\frac{\sum_{k=1}^K\mathbb E\|\mathcal G_k\|^2}{K^{1/3}}\right)^{1/3}.\]
	Hence, we obtain
	\[
	\frac{\sum_{k=1}^K\mathbb E\|\mathcal G_k\|^2}{K^{1/3}}\le(16\mathcal{C}_1c_\alpha^{-1/3})^{3/2}.
	\]
	Combining the two cases yields
	\[
	\sum_{k=1}^K\mathbb E\|\mathcal G_k\|^2\leq 16(\mathcal{C}_1+\mathcal{C}_2+4\mathcal{C}_1^{3/2}c_\alpha^{-1/2})K^{1/3}=C_{\mathcal G} K^{1/3},
	\]
	which is the first claim. By \eqref{eq:ESK1/3}, substituting  into the bound for \(\mathbb E[D_K]\) gives
	\[
	\mathbb E [D_K^{1/3}]
	\leq \left(D_0+1+\frac{C_{\mathcal G}}{c_\alpha}\right)^{1/3}K^{1/3}=\left(1+\frac{C_{\mathcal G}}{c_\alpha}\right)^{1/3} K^{1/3}.
	\]
	By Lemma \ref{lem:error-sum}, we have
	\[
	\begin{aligned}
		\sum_{k=1}^K \mathbb E\|v_k-\operatorname{grad}F(x_k)\|^2
		\leq&\frac{18c_\alpha c_\eta^2 \widetilde L^2}{3c_\alpha-2}\mathbb E[D_K^{1/3}]
		+\frac{3 \sigma^2(1+6c_\alpha^2)}{3c_\alpha-2} K^{1/3}\\
		\leq &\left(\frac{18c_\alpha c_\eta^2 \widetilde L^2}{3c_\alpha-2}\left(1+\frac{C_{\mathcal G}}{c_\alpha}\right)^{1/3}+\frac{3 \sigma^2(1+6c_\alpha^2)}{3c_\alpha-2}\right) K^{1/3},
	\end{aligned}
	\]
	which proves the second claim. 
	Finally, by Lemma \ref{lem:tracking}, we obtain
	\[\begin{aligned}
		\sum_{k=1}^K\mathbb E\|\lambda_k+w_{\rho_k}(x_k)\|^2\leq &
		\frac{4c_\rho^2 c_\eta^2\tau^2\varrho^2\|A\|^2(1+\tau^2) }{(1-\tau)^2(1-\tau^2)^2}\sum_{k=1}^K\mathbb E\|\mathcal G_k\|^2
		+\frac{8L_g^2}{1-\tau^2}+\frac{2L_g^2\pi^2(1+\tau^2)}{27(1-\tau^2)^2}\\
		\leq &\left( \frac{4c_\rho^2 c_\eta^2\tau^2\varrho^2\|A\|^2(1+\tau^2)}{(1-\tau)^2(1-\tau^2)^2} C_{\mathcal G} +\frac{8L_g^2}{1-\tau^2}+\frac{2L_g^2\pi^2(1+\tau^2)}{27(1-\tau^2)^2}\right) K^{1/3},
	\end{aligned}\]
	which proves the third claim and completes the proof.
\end{proof}

Finally, we use the above inequalities to prove the main convergence result for Algorithm \ref{alg1}. 
\begin{theorem}
	\label{the:complexity}
	Suppose that Assumptions \ref{ass1}--\ref{ass3} hold, and consider the sequence $\{x_k, y_k, \lambda_k\}_1^K$ generated by Algorithm \ref{alg1}.
	Set $\frac{2}{3}<c_\alpha\leq 1$, $c_\rho>0$ and $c_\eta\leq\min\left\{\frac{1}{2M_\Phi c_\rho},\frac{(1-\tau)(1-\tau^2)}{4 \varrho\tau^2\|A\|^2\sqrt{2(1+\tau^2)}}\right\}\frac{1}{c_\rho}$. Define $M_F:=\max_{x\in\mathcal M}\|\nabla F(x)\|$. Let the constants $C_{\mathcal G}$ $C_v$, $\varrho$ and $L_{\mathcal P}$ be those given in Theorem \ref{the:sumG} and Lemmas \ref{proposi:ret-lips} and \ref{lem:projector-lipschitz}, respectively, and \(J\) be the tail-uniform output index of Algorithm \ref{alg1}. Then we have
	\begin{align}
		-\lambda_{J+1}&\in\partial g(y_{J+1}),\label{eq:main-graph}\\
		\|Ax_{J+1}-y_{J+1}\|&\leq\frac{2^{1/3}(1+\tau)L_g}{c_\rho K^{1/3}},\label{eq:main-feas}\\
		\mathbb E\left\|\mathcal P_{T_{x_{J+1}}\mathcal M}(\nabla F(x_{J+1})-A^\top\lambda_{J+1})\right\|
		&\leq\frac{2\sqrt{ C_{\mathcal G} C_F^2+C_v}}{K^{1/3}},\label{eq:main-stat}
	\end{align}
	where $C_F:=1+\varrho c_\eta(L_F+L_{\mathcal P}(M_F+\|A\|L_g)+c_\rho \|A\|^2)$. 
	Consequently, MARS-ADMM finds an (expected) \(\epsilon\)-KKT triple with iteration, SFO, independent sample, proximal, and retraction complexities all equal to \(\mathcal O(\epsilon^{-3})\).
\end{theorem}
\begin{proof}
	Initialization yields $-\lambda_1\in\operatorname{dom}g^*$, and every proximal image also belongs to this domain. Hence
	\[
	-\lambda_k\in\operatorname{dom}g^*,\quad
	-\widehat\lambda_{k+1}\in\operatorname{dom}g^*,\quad 
	\|\lambda_k\|\leq L_g,\quad 
	\|\widehat\lambda_{k+1}\|\leq L_g.
	\]
	Then the optimality condition of \eqref{eq:y-corrector} proves \eqref{eq:main-graph} pathwise. 
	By $\rho_k=c_\rho k^{1/3}$, we obtain
	\[
	\|Ax_{k+1}-y_{k+1}\|
	\overset{\eqref{eq:lambda-full}}{\leq}\frac{\tau\|\lambda_k\|+\|\lambda_{k+1}\|}{\rho_k}
	\leq\frac{(1+\tau)L_g}{c_\rho k^{1/3}}.
	\]
	Since \(J\ge\lceil K/2\rceil\), this proves \eqref{eq:main-feas}. 
	
	The two dual images $-\widehat\lambda_{k+1}=T_{\rho_k,x_k}(-\lambda_k)$ and $-\lambda_{k+1}=T_{\rho_k,x_{k+1}}(-\lambda_k)$ have the same incoming vector and differ only in their base points. The nonexpansiveness of the proximal map yields
	\[
	\|\widehat\lambda_{k+1}-\lambda_{k+1}\|
	=\|\operatorname{prox}_{\rho_k g^*}(\rho Ax_{k+1}-\tau \lambda_k)-\operatorname{prox}_{\rho_k g^*}(\rho Ax_k-\tau \lambda_k)\|
	\leq \rho_k\|A\|\|x_{k+1}-x_k\|.
	\]
	Both sides of the stationarity comparison are vectors in \(\mathbb R^n\). Define $M_F:=\max_{x\in\mathcal M}\|\nabla F(x)\|$. Splitting the change, it follows Lemma \ref{lem:projector-lipschitz} that
	\[
	\begin{aligned}
		&\|\mathcal P_{T_{x_{k+1}}\mathcal M}(\nabla F(x_{k+1})-A^\top \lambda_{k+1})
		-\mathcal P_{T_{x_k}\mathcal M}(\nabla F(x_k)-A^\top \widehat\lambda_{k+1})\|\\
		\leq& \|\mathcal P_{T_{x_{k+1}}\mathcal M}(\nabla F(x_{k+1})-\nabla F(x_k)+A^\top \widehat\lambda_{k+1}-A^\top \lambda_{k+1})\|\\
		&+\|\mathcal P_{T_{x_{k+1}}\mathcal M}(\nabla F(x_k)-A^\top \widehat\lambda_{k+1})
		-\mathcal P_{T_{x_k}\mathcal M}(\nabla F(x_k)-A^\top \widehat\lambda_{k+1})\|\\
		\leq&\|\nabla F(x_{k+1})-\nabla F(x_k)\|+\|A\|\|\widehat\lambda_{k+1}-\lambda_{k+1}\|
		+\|\mathcal P_{T_{x_{k+1}}\mathcal M}-\mathcal P_{T_{x_k}\mathcal M}\| \|\nabla F(x_k)-A^\top \widehat\lambda_{k+1}\|\\
		\leq&(L_F+L_{\mathcal P}(M_F+\|A\|L_g)) \|x_{k+1}-x_k\|+\|A\|\|\widehat\lambda_{k+1}-\lambda_{k+1}\|\\
		\leq&(L_F+L_{\mathcal P}(M_F+\|A\|L_g)+\rho_k\|A\|^2)\|x_{k+1}-x_k\|.
	\end{aligned}
	\]
	According to Lemma \ref{proposi:ret-lips}, we have that $\|x_{k+1}-x_k\|\leq \varrho\eta_k\|\mathcal G_k\|$. 
	By using \(\eta_k\leq c_\eta k^{-1/3} \leq c_\eta\) and \(\rho_k\eta_k\le c_\rho c_\eta\), we obtain
	\[\begin{aligned}
		&\|\mathcal P_{T_{x_{k+1}}\mathcal M}(\nabla F(x_{k+1})-A^\top\lambda_{k+1})\|\\
		\leq &\|\mathcal P_{T_{x_{k+1}}\mathcal M}(\nabla F(x_{k+1})-A^\top \lambda_{k+1})
		-\mathcal P_{T_{x_k}\mathcal M}(\nabla F(x_k)-A^\top \widehat\lambda_{k+1})\|
		+\|\mathcal P_{T_{x_k}\mathcal M}(\nabla F(x_k)-A^\top \widehat\lambda_{k+1})\|\\
		\leq& \varrho c_\eta(L_F+L_{\mathcal P}(M_F+\|A\|L_g)+c_\rho \|A\|^2) \|\mathcal G_k\|+\|\mathcal G_k\|+\|v_k-\operatorname{grad}F(x_k)\|\\
		=& C_F\|\mathcal G_k\|+\|v_k-\operatorname{grad}F(x_k)\|.
	\end{aligned}\]
	Thus,
	\[
	\|\mathcal P_{T_{x_{k+1}}\mathcal M}(\nabla F(x_{k+1})-A^\top\lambda_{k+1})\|^2
	\le2C_F^2\|\mathcal G_k\|^2+2\|v_k-\operatorname{grad}F(x_k)\|^2.
	\]
	Conditionally on the generated trajectory, \(J\) is uniform on a set of at least \(K/2\) indices. Finally, it follows Theorem \ref{the:sumG} that
	\[
	\begin{aligned}
		\mathbb E\|\mathcal P_{T_{x_{J+1}}\mathcal M}(\nabla F(x_{J+1})-A^\top\lambda_{J+1})\|^2
		\leq&\frac4K \sum_{k=1}^K \left(C_F^2\mathbb E\|\mathcal G_k\|^2+\mathbb E\|v_k-\operatorname{grad}F(x_k)\|^2\right)\\
		\leq& 4(C_F^2 C_{\mathcal G}+C_v)K^{-2/3}.
	\end{aligned}
	\]
	Jensen's inequality converts it to prove \eqref{eq:main-stat}. Together with \eqref{eq:main-graph}--\eqref{eq:main-feas}, choosing
	\[
	K\ge
	\max\left\{
	\left(\frac{2\sqrt{C_{\mathcal G} C_F^2+C_v}}{\epsilon}\right)^3,~
	\left(\frac{2^{1/3}(1+\tau)L_g}{c_\rho\epsilon}\right)^3,~
	3
	\right\},
	\]
	produces an (expected) \(\epsilon\)-KKT triple. Here $K\geq 3$ because $\lceil K/2\rceil +1\leq J+1\leq K$.
	
	The initialization $\lambda_1=\operatorname{prox}_g(0)$ requires one ordinary proximal evaluation. Each of the \(K\) stages uses two proximal evaluations for the $y$-update and one retraction for the $x$-update. Therefore, the exact proximal count is \(2K+1=\mathcal O(K)\). The exact SFO count is
	\[
	|\mathcal{S}_1|+2\sum_{k=2}^K|\mathcal{S}_k|\le(2K-1)|\mathcal{S}_{\max}|=\mathcal O(K).
	\]
	because every noninitial common batch is queried at two points. The number of independent samples is \(\sum_{k=1}^K|\mathcal{S}_k|=\mathcal O(K)\). This proves all complexity statements.
\end{proof}

According to Theorem \ref{the:complexity}, our stochastic variant has the same complexity $\mathcal{O}(\epsilon^{-3})$ as the deterministic counterpart described in \cite{dengadaptive}. 
This improvement is due to the single-loop structure of MARS-ADMM, in which the stepsize $\eta_k$ is coupled to the penalty parameter $\rho_k$ as $\rho_k\eta_k \leq c_\rho c_\eta =\mathcal{O}(1)$. This coupling in \eqref{eq:x-full} stabilizes the update magnitude $\|x_{k+1}-x_k\|\approx\|\eta_k \mathcal{G}_k\|\approx \mathcal{O}(1)$, thereby neutralizing the adverse effect of an increasing $\rho_k$. In contrast, the condition number of the inner solver in StoManIAL \cite{deng2024oracle} worsens as its penalty parameter increases. This results in a lower complexity $\tilde{\mathcal{O}}(\epsilon^{-3.5})$ than the $\mathcal{O}(\epsilon^{-3})$ attained by its deterministic counterpart.

\section{Experiments}\label{sec5}
This section reports the numerical results of our proposed MARS-ADMM algorithm. All experiments are performed in MATLAB R2023b on a 64-bit laptop equipped with Intel i9-13900HX CPU and 32.0 GB RAM.  
We compare MARS-ADMM with the nested stochastic Riemannian primal-dual method StoManIAL-II (Option II of \cite{deng2024oracle}), the Riemannian stochastic subgradient (R-Subgrad) \cite{li2021weakly}, and two proximal methods R-ProxSGD and R-ProxSPB \cite{wang2022riemannian}.  

\begin{figure*}[]
	\centering
	\includegraphics[width=0.8\linewidth]{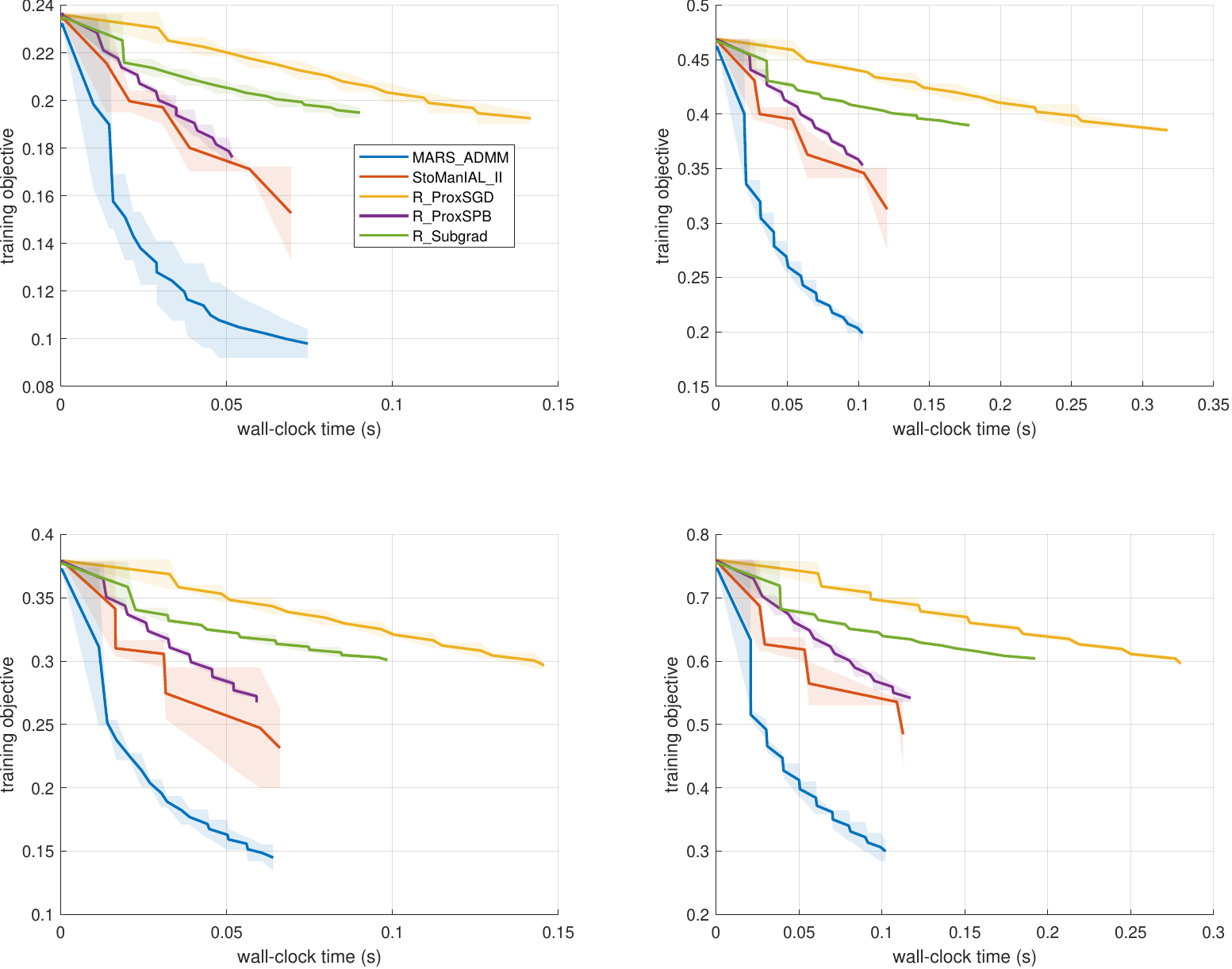}
	\caption{Experimental results on the SPCA problem \eqref{spca} with synthetic datasets ($\mu=0.01$). From top to bottom: $n=80,160$. From left to right: $p=4,8$. }
	\label{fig:spca-syn-mu08-wct}
\end{figure*}

\subsection{Sparse Principal Component Analysis}
Sparse principal component analysis (SPCA) \cite{journee2010generalized,yang2015streaming} is a key technique for high-dimensional data analysis, identifying principal components with sparse loadings. For normalized data vectors $z_i\in\mathbb R^n$, put $H=\frac{1}{N}\sum_{i=1}^N z_i z_i^\top$.  Up to an additive constant in the mean reconstruction error, the training model is
\begin{equation}\label{spca}
	\min_{X\in {\rm St}(n,p)}\{-{\rm tr}(X^\top H X)+\mu\|X\|_1\}.
\end{equation}
where ${\rm St}(n, p) := \{X\in\mathbb{R}^{n\times p}:X^\top X=I_p\}$ is the Stiefel manifold. This problem satisfies the smoothness and Lipschitz conditions used in the theory. In fact, $\|z_i\|=1$ implies $\|H\|\leq1$, so the ambient gradient of $F(X)=-\operatorname{tr}(X^\top HX)$ is 2-Lipschitz in the Frobenius norm.  Furthermore, $g(X)=\mu\|X\|_1$ is globally $L_g$-Lipschitz in the Frobenius norm with $L_g=\mu\sqrt{np}$. The compactness of the Stiefel manifold makes the finite-sum stochastic variance uniformly bounded.

The confirmatory experiment contains three paired repetitions with the disjoint seeds. Within a repetition all methods use the same initialization and are run under an SFO cap of 15000. 
Since some of the problem-dependent constants cannot be directly estimated from the datasets, a grid search is performed to tune the parameters of all algorithms. 
The search covers $(c_\rho,c_\eta,\tau)\in\{0.1,0.5,1,1.5,2\}\times\{0.001,0.005,0.01,0.05,0.1,0.2,0.4\} \times\{0.3,0.6,0.9\}$ and $c_\alpha=0.8$ for MARS-ADMM, $(\sigma_0,\beta_1,\kappa)\in\{0.1,0.5,1\}\times\{0.1,0.5,1\} \times\{0.01,0.05,0.1\}$ for StoManIAL-II, and $\eta_0\in\{0.001,0.005,0.01,0.05,0.1\}$ for R-Subgrad, R-ProxSGD and R-ProxSPB. 
For the batch size settings, we use the online R-ProxSPB variant, with a snapshot batch of 256, a recursive batch of 32, and a refresh period 8. Every gradient in either batch is charged as an SFO call. All the other methods use a mini-batch size of 32.

We test the full factorial grid $(n,p,\mu)\in\{80,160\}\times\{4,8\} \times \{0.01,0.1\}$, giving 8 synthetic instances. Each instance has $N=1500$ normalized samples. 
We report the empirical training objective against solver wall-clock time as the primary optimization measure. 
The experimental results are shown in Figures \ref{fig:spca-syn-mu08-wct}--\ref{fig:spca-syn-mu12-wct} with one panel for each $(n,p)$ configuration. For each method, we plot the mean over three repetitions and a 95\% bootstrap band computed by resampling the three paired runs.

We next evaluate the SPCA problem \eqref{spca} on two real datasets: COIL-100 \cite{nene1996columbia}, using 7,200 images resized to $32\times32$ with $n=1024$, and MNIST \cite{lecun1998mnist}, using 60,000 images with $n=784$. 
Each image column is normalized before splitting. For real data, a fixed seed-independent 80/20 partition is made once: it is stratified by label for classification and random for SPCA. 
Figure \ref{fig:spca-real-wct} reports the wall-clock objective trajectories.
These results show that the MARS-ADMM method reach the same objective range much earlier in wall-clock solver time than the nested and tangent-proximal baselines.

\begin{figure*}[]
	\centering
	\includegraphics[width=0.8\linewidth]{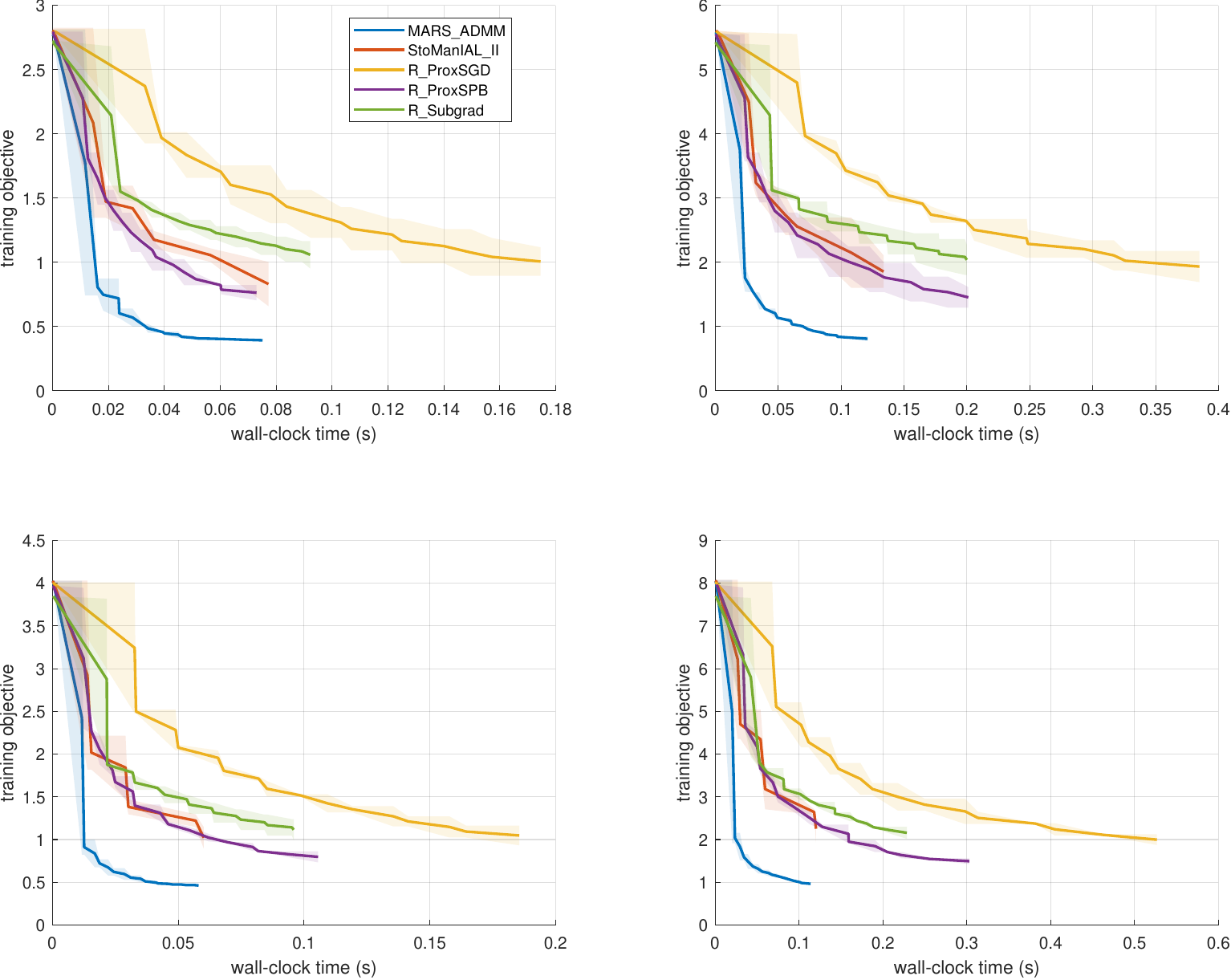}
	\caption{Experimental results on the SPCA problem \eqref{spca} with synthetic datasets ($\mu=0.1$). From top to bottom: $n=80,160$. From left to right: $p=4,8$.}
	\label{fig:spca-syn-mu12-wct}
\end{figure*}

\begin{figure*}[]
	\centering
	\begin{subfigure}{0.4\linewidth}
		\centering
		\includegraphics[width=\linewidth]{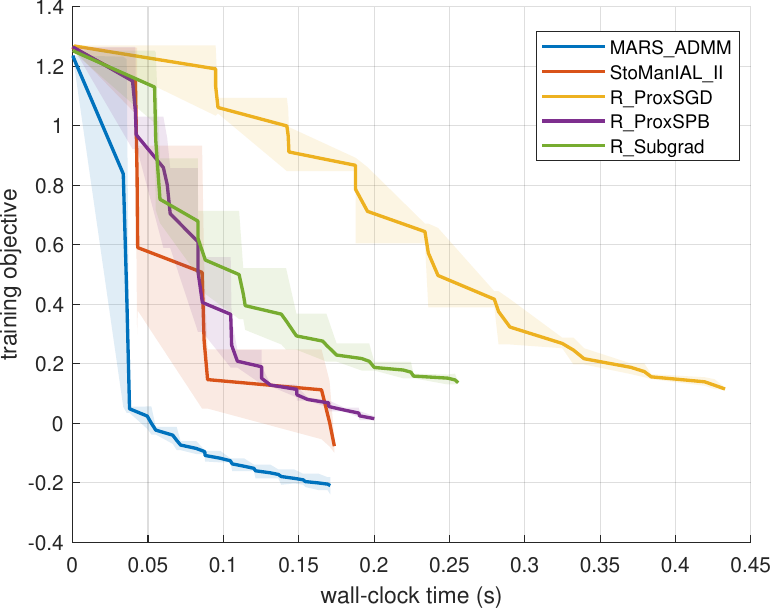}
		\includegraphics[width=\linewidth]{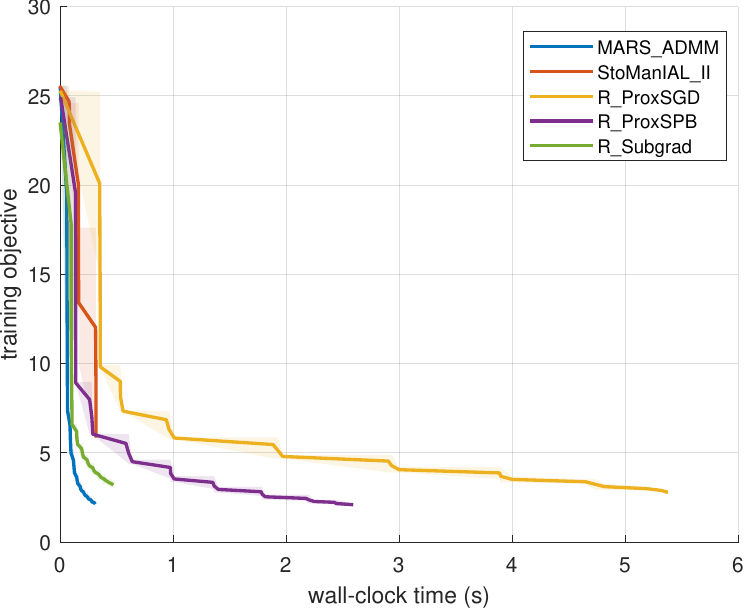}
		\caption{COIL-100}
	\end{subfigure}
	\begin{subfigure}{0.4\linewidth}
		\centering
		\includegraphics[width=\linewidth]{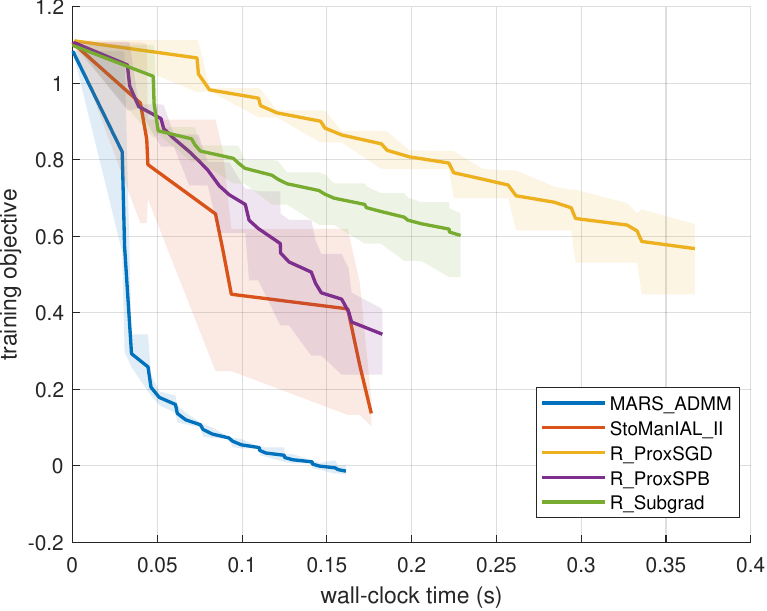}
		\includegraphics[width=\linewidth]{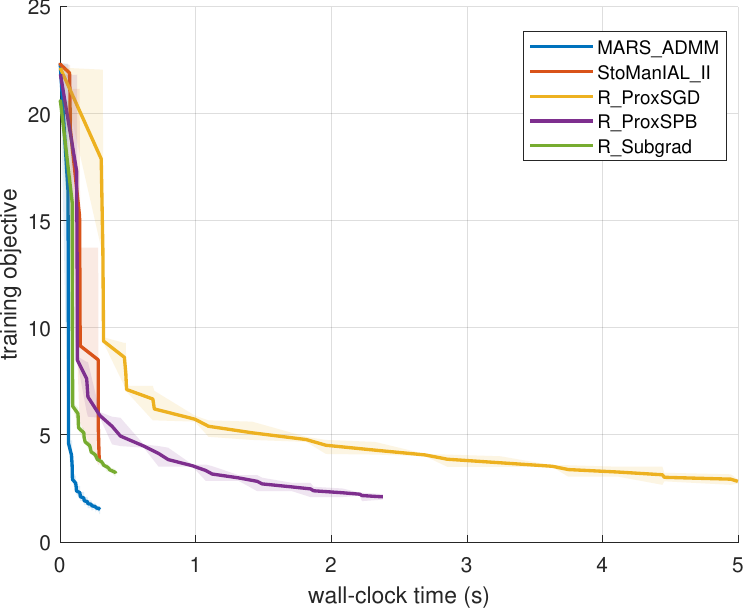}
		\caption{MNIST}
	\end{subfigure}
	\caption{Experimental results on the SPCA problem \eqref{spca} with real datasets. The first row: $p=5$ and $\mu=0.01$, and the second row: $p=10$ and $\mu=0.1$.}\label{fig:spca-real-wct}
\end{figure*}

\subsection{Regularized Linear Classification over the Sphere}
Optimization over sphere manifolds has been motivated by the Hadamard parametrization of simplex constraints \cite{li2023simplex,levin2025effect}. Simplex constraints arise in empirical risk minimization when the model parameters represent probability distributions, proportions, or normalized weights \cite{li2020methods}. 
Consider a classification problem with $N$ training examples $\{a_i,b_i\}_{i=1}^N$, where $a_i\in\mathbb{R}^{m \times 1}$ and $b_i\in\{-1,1\}$ for all $i \in [N]$. The goal is to minimize a smooth nonconvex loss function \cite{zhang2023riemannian} with $\ell_1$-regularization
\begin{equation}\label{rlcs}
	\min_{x\in\mathbb S^{m-1}}\left\{\frac1N\sum_{i=1}^N\left(1-\frac{1}{1+\exp(-b_i a_i^\top x)}\right)^2+\mu\|x\|_1\right\},
\end{equation}
over the sphere manifold, $x \in {\rm S}^{m-1}:=\{x\in\mathbb{R}^m:~x^\top x=1\}$. 
This model, denoted RLCS, satisfies the smoothness assumptions. For $h(t)=(1-(1+\exp(-t))^{-1})^2$, one has $|h''(t)|\le1$. Since every feature vector is normalized, each sample loss $h(b_i a_i^\top x)$ and the empirical loss have 1-Lipschitz Euclidean gradients. The regularizer $g(x)=\mu\|x\|_1$ is globally $L_g=\mu\sqrt m$-Lipschitz in the Euclidean norm. On the compact sphere, these uniformly smooth finite-sum sample losses satisfy the requirements of bounded variance and mean-squared retraction smoothness.

We conduct tests on three real datasets: a9a ($m=123,N=32561$), phishing ($m=68,N=11055$), and w8a ($m=300,N=49749$) from LIBSVM \cite{chang2011libsvm}. In these experiments, the MARS-ADMM parameters are set to $(c_\rho,c_\eta,\tau)=(0.25,0.5,0.8)$. Figures \ref{fig:rlc-cpu}--\ref{fig:rlc-accuracy} report the computational cost and held-out classification test accuracy.  
From the figures we observe that for these datasets, MARS-ADMM reduces the objective function value faster than the other algorithms, while achieving similarly competitive test accuracy.Additionally, MARS-ADMM and StoManIAL-II substantially outperform subgradient and proximal methods, suggesting that the primal-dual framework enhances numerical performance for these instances.

We additionally evaluate the offline primal-dual certificate
\[
\mathcal E(x,y,\lambda)=\max\{\|\mathcal P_{T_x\mathcal M}
(\nabla F(x)-A^\top\lambda)\|,\ 
{\rm dist}(-\lambda,\partial g(y)),\
\|Ax-y\|\},
\]
for the two primal-dual methods. 
Figures \ref{fig:rlc-kkt}--\ref{fig:rlc-feas} show the KKT violation and primal-feasibility for MARS-ADMM and StoManIAL-II. 
The observations from the figures indicate that, within the same number of stochastic gradient evaluations, MARS-ADMM reduces the constraint residual much faster throughout the algorithm's progress.

\begin{figure*}[]
	\centering
	\begin{subfigure}{.3\linewidth}
		\centering
		\includegraphics[width=\linewidth]{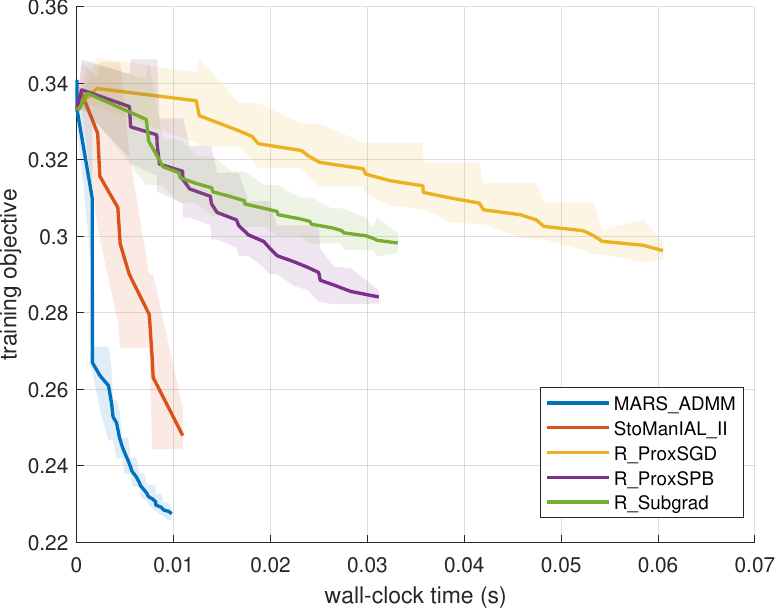}
		\includegraphics[width=\linewidth]{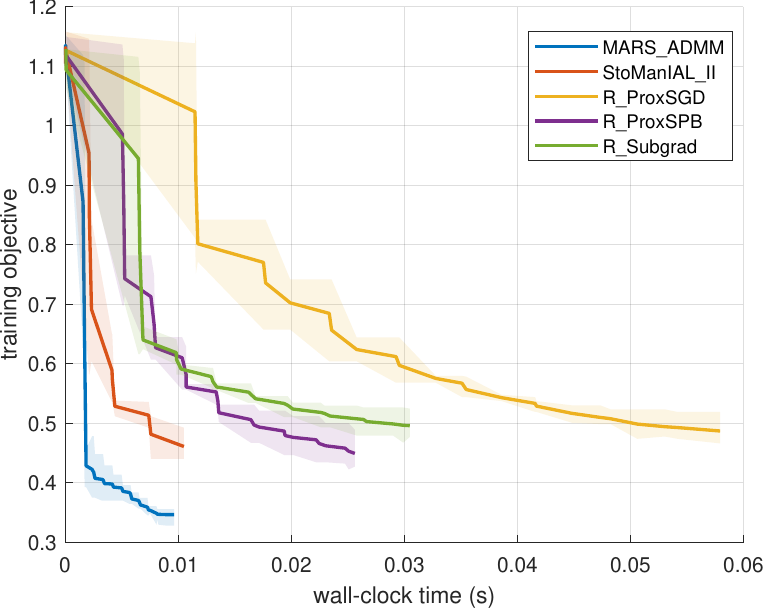}
		\caption{a9a}
	\end{subfigure}
	\begin{subfigure}{.3\linewidth}
		\centering
		\includegraphics[width=\linewidth]{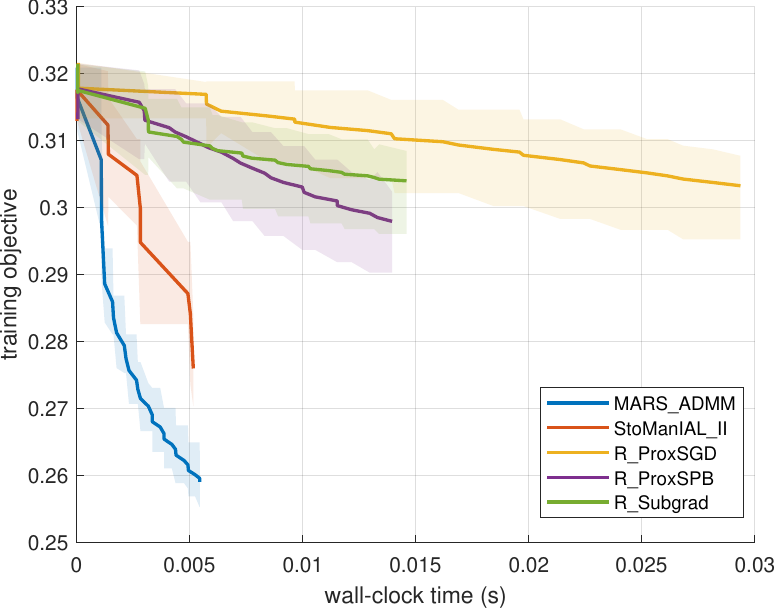}
		\includegraphics[width=\linewidth]{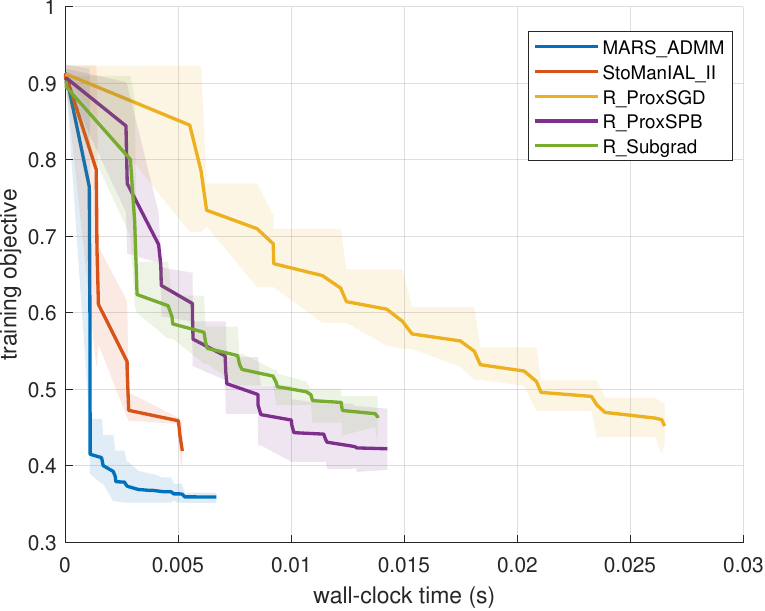}
		\caption{phishing}
	\end{subfigure}
	\begin{subfigure}{.3\linewidth}
		\centering
		\includegraphics[width=\linewidth]{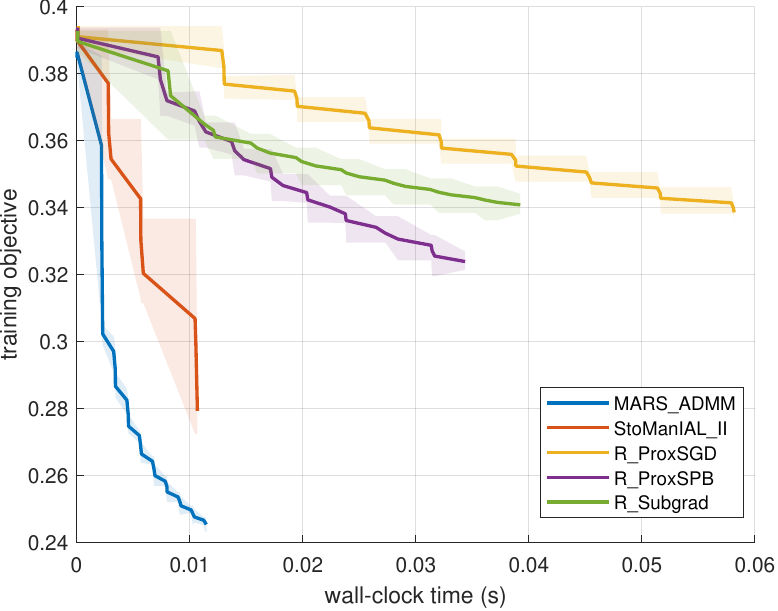}
		\includegraphics[width=\linewidth]{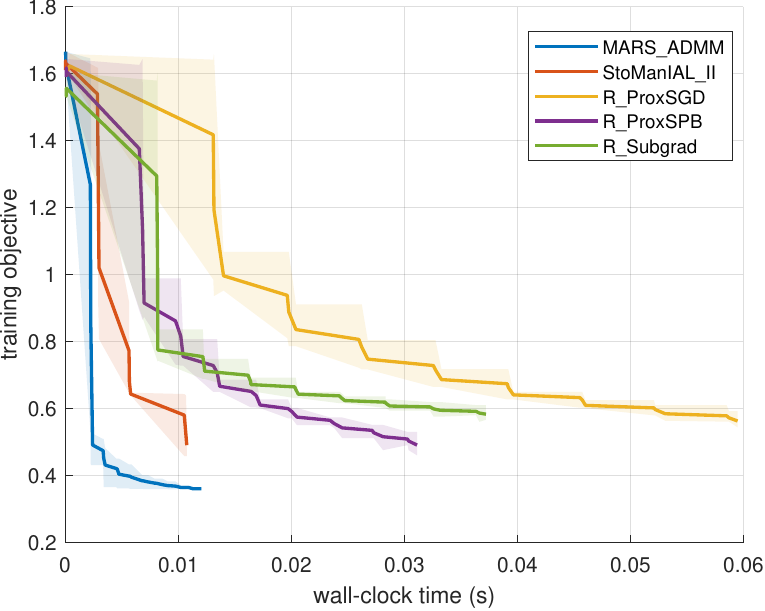}
		\caption{w8a}
	\end{subfigure}
	\caption{Experimental results on the RLCS problem \eqref{rlcs}. The first row: $\mu=0.01$, and the second row: $\mu=0.1$.}\label{fig:rlc-cpu}
\end{figure*}

\begin{figure*}[]
	\centering
	\begin{subfigure}{.3\linewidth}
		\centering
		\includegraphics[width=\linewidth]{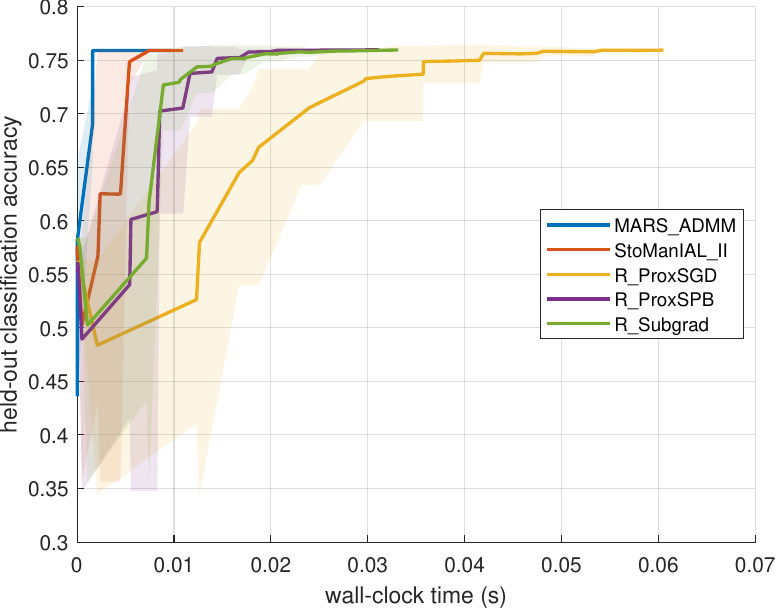}
		\includegraphics[width=\linewidth]{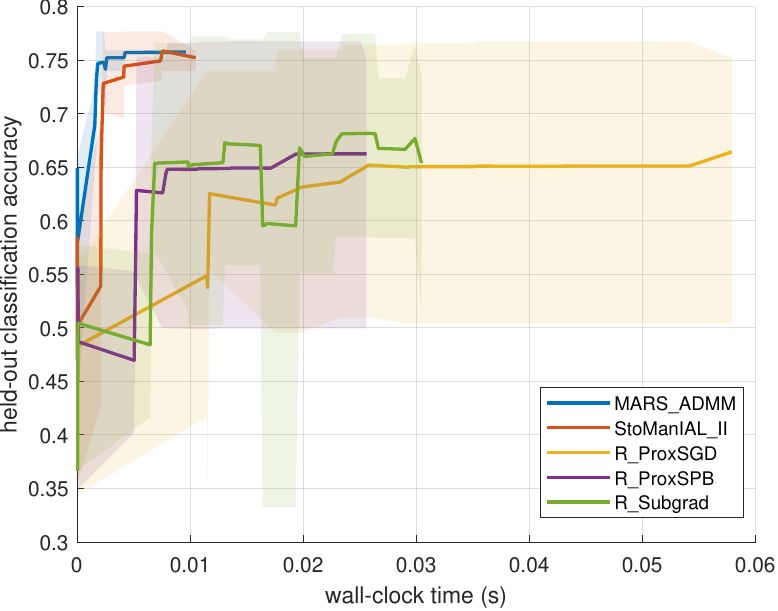}
		\caption{a9a}
	\end{subfigure}
	\begin{subfigure}{.3\linewidth}
		\centering
		\includegraphics[width=\linewidth]{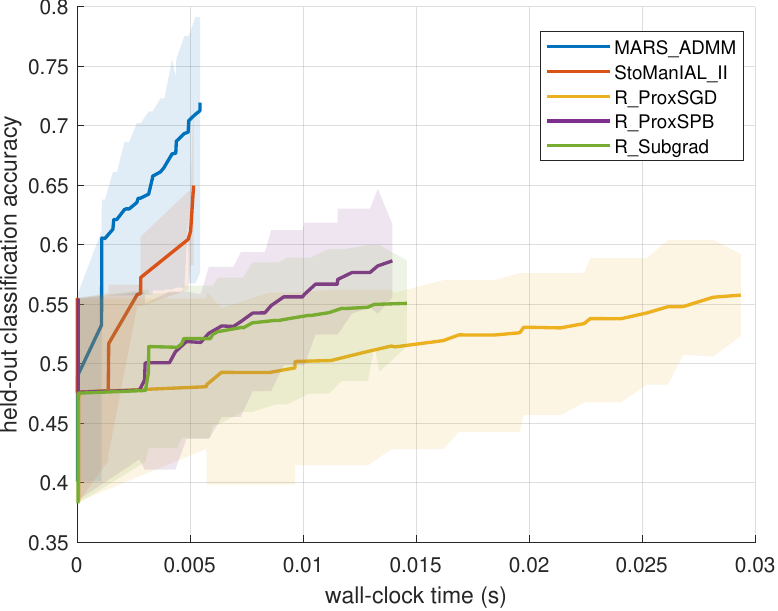}
		\includegraphics[width=\linewidth]{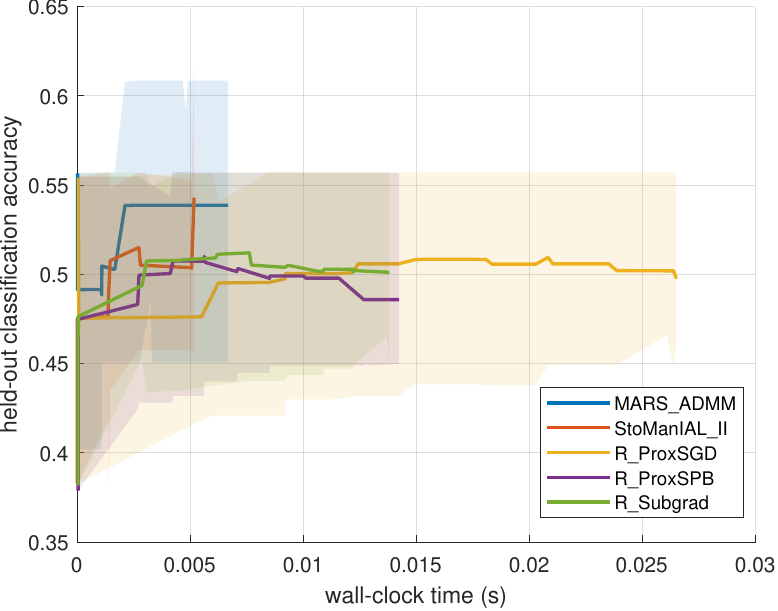}
		\caption{phishing}
	\end{subfigure}
	\begin{subfigure}{.3\linewidth}
		\centering
		\includegraphics[width=\linewidth]{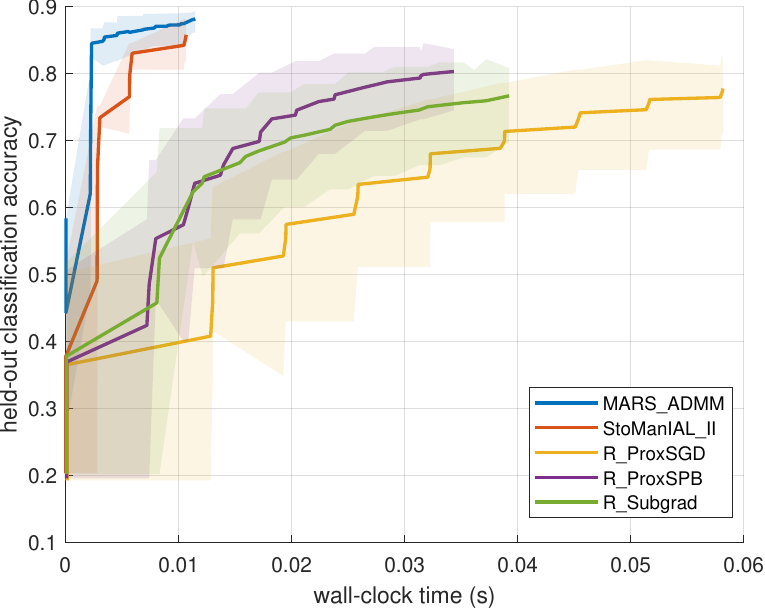}
		\includegraphics[width=\linewidth]{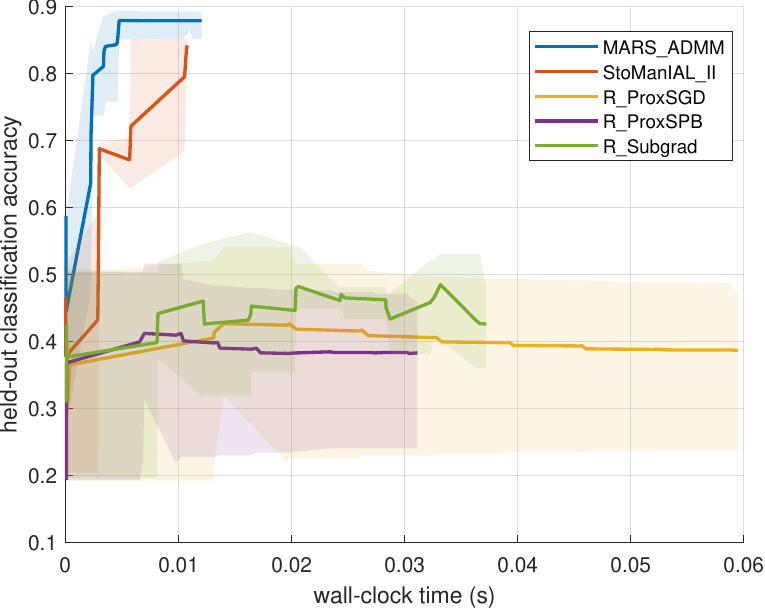}
		\caption{w8a}
	\end{subfigure}
	\caption{Classification accuracy of the RLCS problem \eqref{rlcs}. The first row: $\mu=0.01$, and the second row: $\mu=0.1$.}\label{fig:rlc-accuracy}
\end{figure*}

\begin{figure*}[]
	\centering
	\begin{subfigure}{.3\linewidth}
		\centering
		\includegraphics[width=\linewidth]{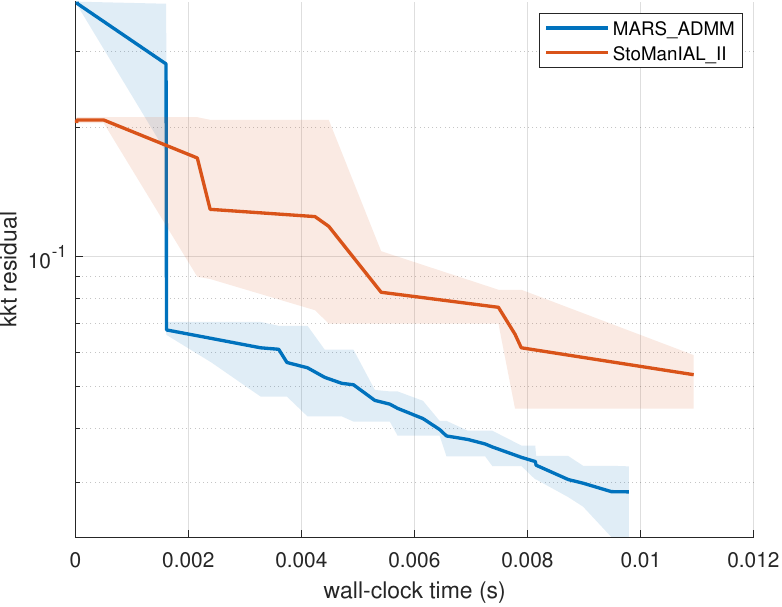}
		\includegraphics[width=\linewidth]{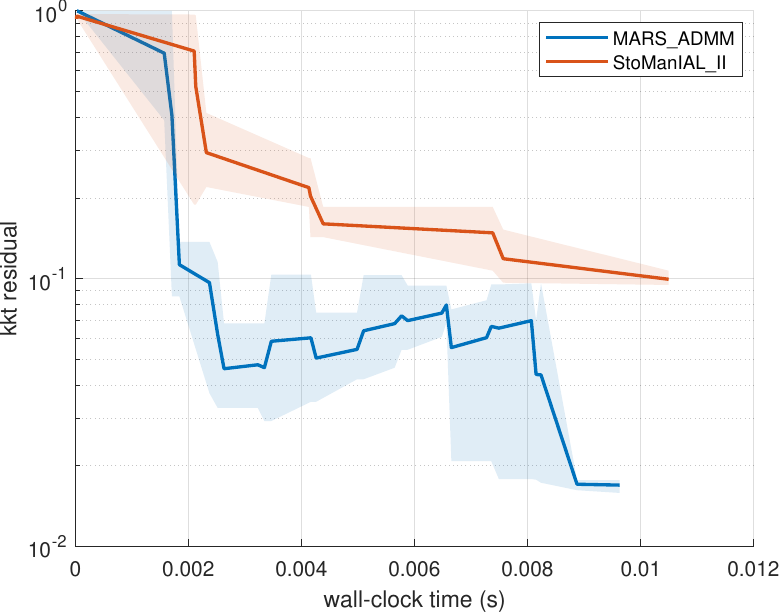}
		\caption{a9a}
	\end{subfigure}
	\begin{subfigure}{.3\linewidth}
		\centering
		\includegraphics[width=\linewidth]{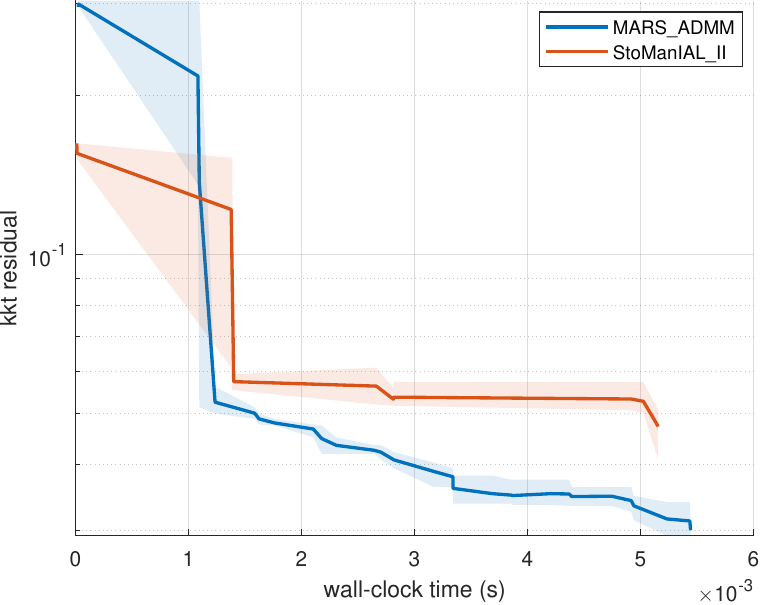}
		\includegraphics[width=\linewidth]{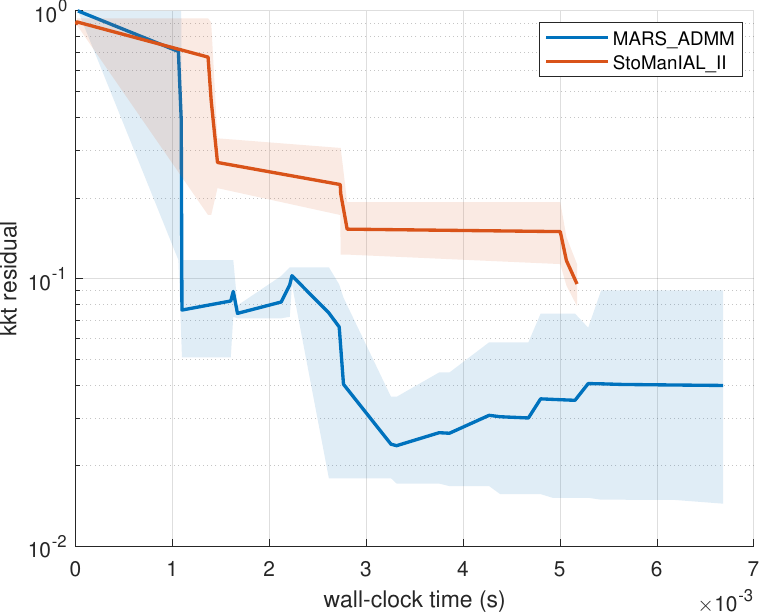}
		\caption{phishing}
	\end{subfigure}
	\begin{subfigure}{.3\linewidth}
		\centering
		\includegraphics[width=\linewidth]{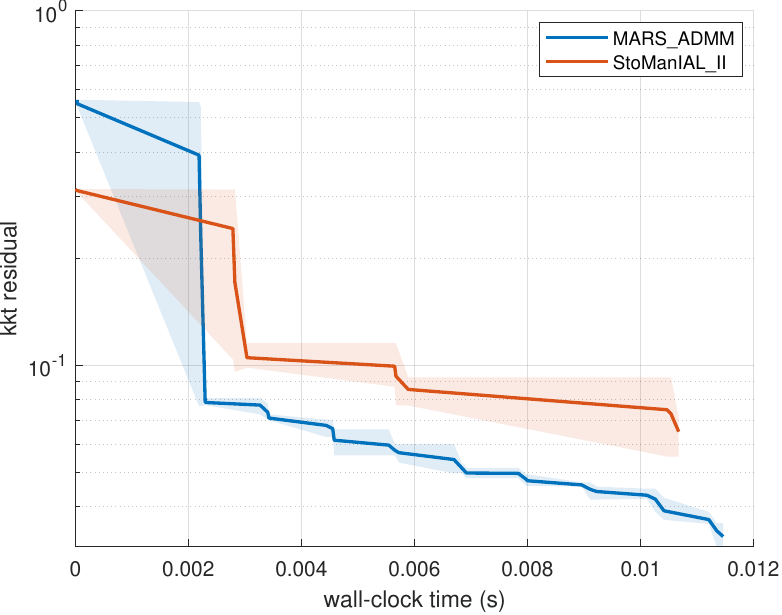}
		\includegraphics[width=\linewidth]{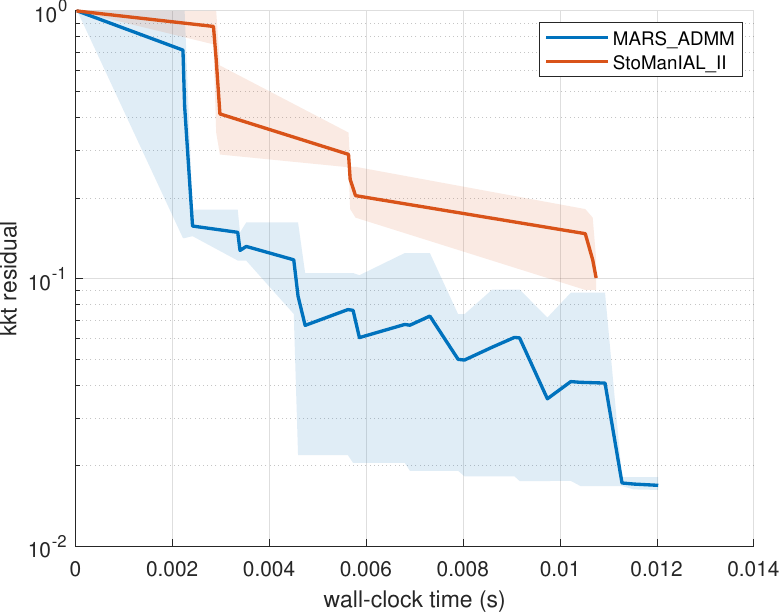}
		\caption{w8a}
	\end{subfigure}
	\caption{KKT violation for MARS-ADMM and StoManIAL-II. The first row: $\mu=0.01$, and the second row: $\mu=0.1$.}\label{fig:rlc-kkt}
\end{figure*}

\begin{figure*}[]
	\centering
	\begin{subfigure}{.3\linewidth}
		\centering
		\includegraphics[width=\linewidth]{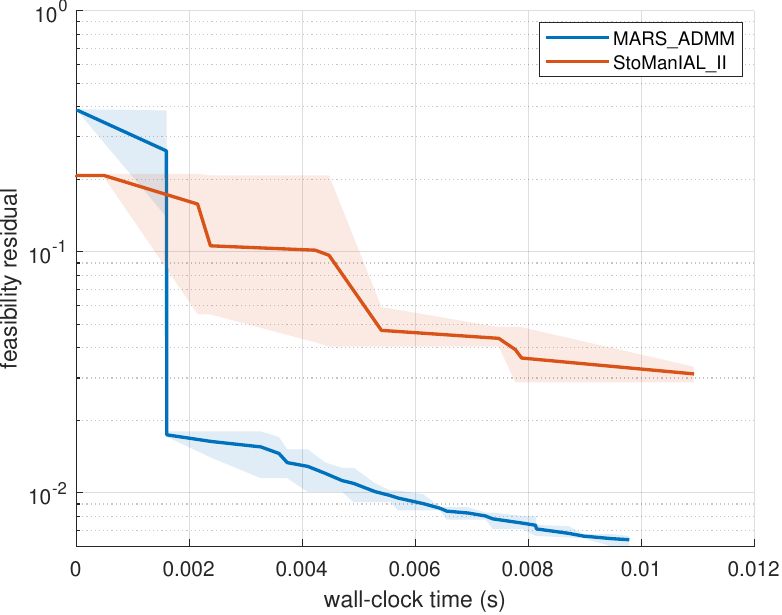}
		\includegraphics[width=\linewidth]{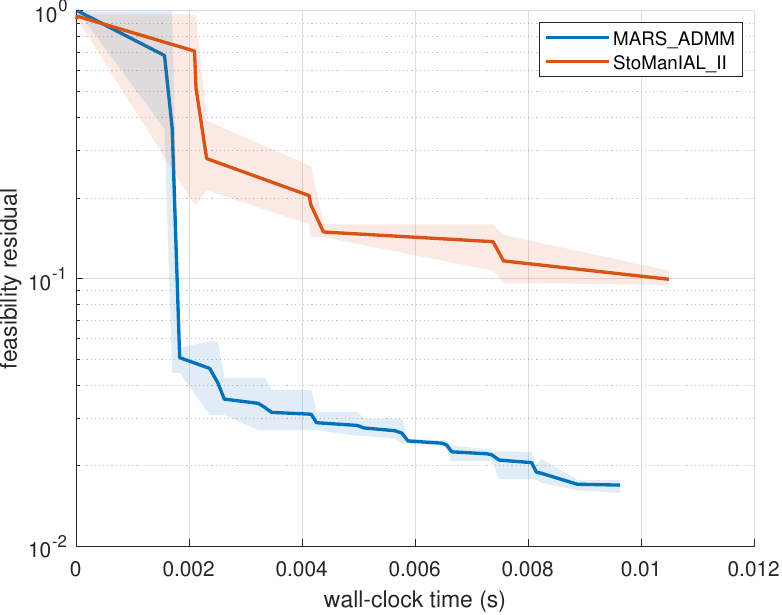}
		\caption{a9a}
	\end{subfigure}
	\begin{subfigure}{.3\linewidth}
		\centering
		\includegraphics[width=\linewidth]{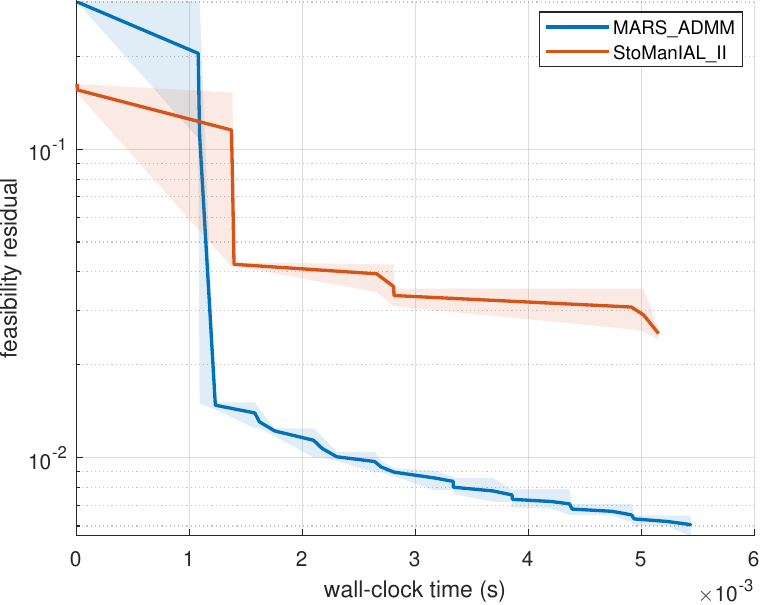}
		\includegraphics[width=\linewidth]{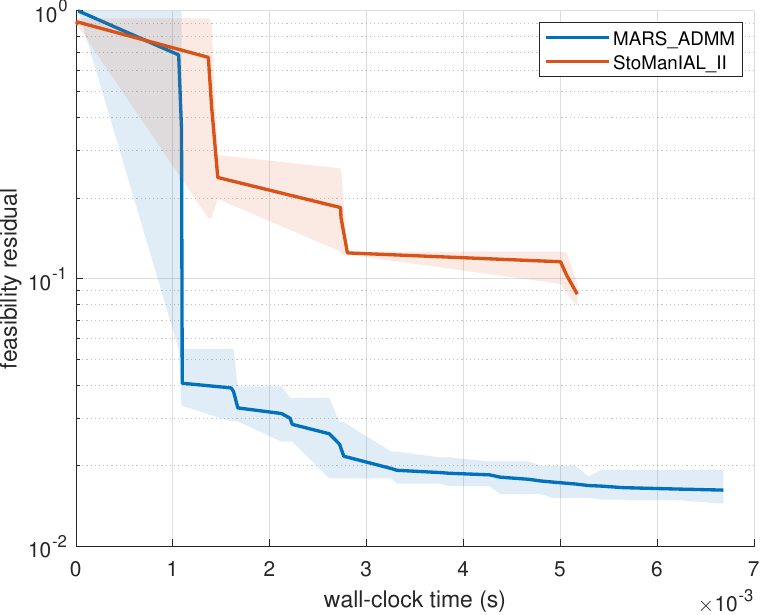}
		\caption{phishing}
	\end{subfigure}
	\begin{subfigure}{.3\linewidth}
		\centering
		\includegraphics[width=\linewidth]{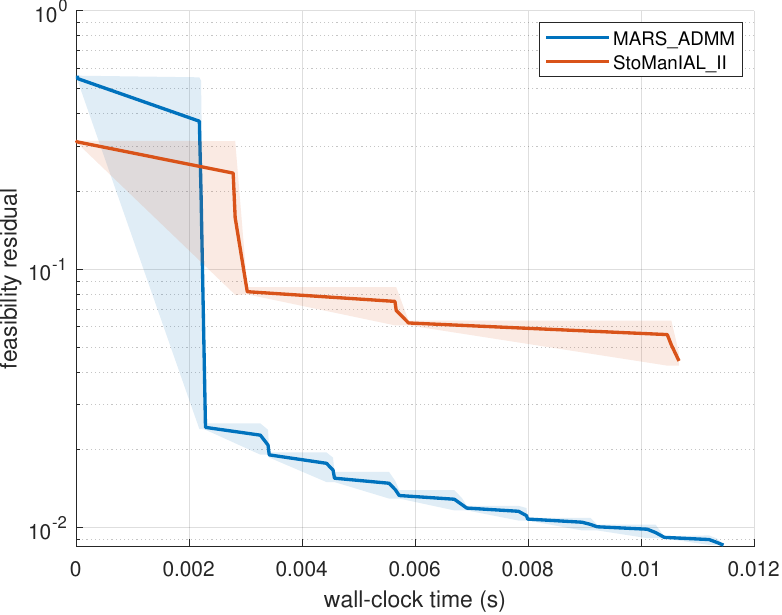}
		\includegraphics[width=\linewidth]{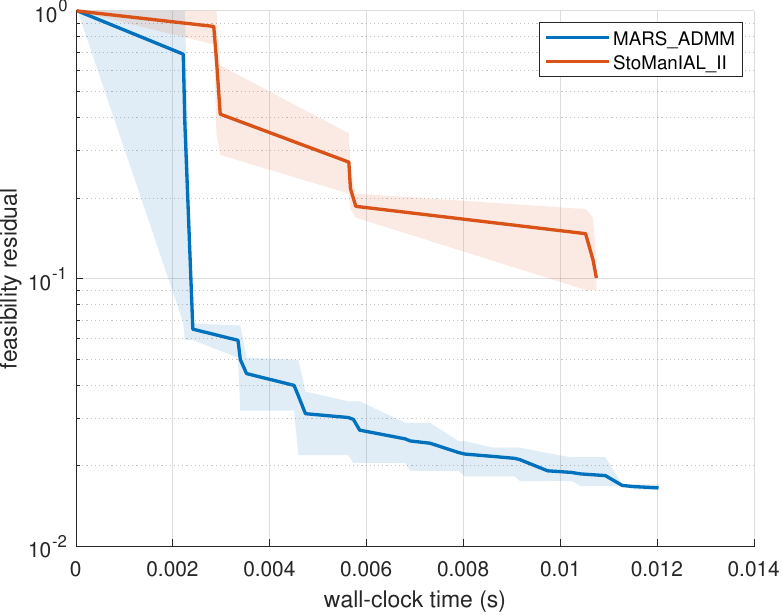}
		\caption{w8a}
	\end{subfigure}
	\caption{Primal feasibility $\|x-y\|$ for MARS-ADMM and StoManIAL-II. The first row: $\mu=0.01$, and the second row: $\mu=0.1$.}\label{fig:rlc-feas}
\end{figure*}

To separate computational speed from the final objective level, we conduct an equal-objective CPU study.  During the calibration stage, we use different seeds and the 15,000-SFO to freeze one target training objective for each instance. This objective is the largest terminal objective over the baseline methods, plus a $10^{-3}$ tolerance. Thus, the proposed method cannot lower the target.  The target stage uses disjoint seeds and frozen targets and stops a trajectory at its first target crossing. Table \ref{tab:cpu-target} reports the resulting net wall-clock seconds. MARS-ADMM reaches the target faster than the baselines on eight out of ten real instances, while R-Subgrad achieves better performance on the remaining instances.

\begin{table*}[]
	\centering
	\resizebox{1\textwidth}{!}{\begin{tabular}{ccccccc}\toprule
			Instance  & Target & MARS-ADMM & StoManIAL-II & R-ProxSGD & R-ProxSPB & R-Subgrad\\
			\midrule
			mnist ($p=5,\mu=0.01$) & 0.681 & 0.029 $\pm$ 0.004 & 0.111 $\pm$ 0.040 & 0.463 $\pm$ 0.065 & 0.146 $\pm$ 0.011 & 0.258 $\pm$ 0.050\\
			mnist ($p=10,\mu=0.1$) & 9.72 & 0.073 $\pm$ 0.003 & 0.343 $\pm$ 0.155 & 0.473 $\pm$ 0.032 & 0.164 $\pm$ 0.012 & 0.054 $\pm$ 0.003\\
			coil100 ($p=5,\mu=0.01$) & 0.196 & 0.037 $\pm$ 0.002 & 0.178 $\pm$ 0.089 & 0.523 $\pm$ 0.037 & 0.163 $\pm$ 0.016 & 0.278 $\pm$ 0.025\\
			coil100 ($p=10,\mu=0.1$) & 12.9 & 0.080 $\pm$ 0.007 & 0.325 $\pm$ 0.147 & 0.305 $\pm$ 0.005 & 0.125 $\pm$ 0.005 & 0.028 $\pm$ 0.001\\
			\hline
			a9a ($\mu=0.1$) & 0.58 & 0.002 $\pm$ 0.000 & 0.007 $\pm$ 0.001 & 0.036 $\pm$ 0.002 & 0.012 $\pm$ 0.001 & 0.016 $\pm$ 0.005\\
			a9a ($\mu=0.01$) & 0.301 & 0.001 $\pm$ 0.000 & 0.007 $\pm$ 0.005 & 0.059 $\pm$ 0.011 & 0.019 $\pm$ 0.005 & 0.027 $\pm$ 0.006\\
			phishing ($\mu=0.1$) & 0.529 & 0.001 $\pm$ 0.000 & 0.005 $\pm$ 0.000 & 0.028 $\pm$ 0.001 & 0.010 $\pm$ 0.001 & 0.012 $\pm$ 0.002\\
			phishing ($\mu=0.01$) & 0.303 & 0.001 $\pm$ 0.001 & 0.005 $\pm$ 0.004 & 0.024 $\pm$ 0.000 & 0.015 $\pm$ 0.007 & 0.008 $\pm$ 0.000\\
			w8a ($\mu=0.1$) & 0.645 & 0.003 $\pm$ 0.000 & 0.018 $\pm$ 0.009 & 0.075 $\pm$ 0.014 & 0.025 $\pm$ 0.005 & 0.034 $\pm$ 0.010\\
			w8a ($\mu=0.01$)& 0.34 & 0.002 $\pm$ 0.000 & 0.011 $\pm$ 0.002 & 0.107 $\pm$ 0.025 & 0.045 $\pm$ 0.011 & 0.061 $\pm$ 0.000\\
			\bottomrule
	\end{tabular}}
	\caption{Solver wall-clock seconds to reach a frozen common training-objective target. Entries are mean $\pm$ standard deviation over three independent target-stage runs. Instances above the line are problem \eqref{spca}, and instances below the line are problem \eqref{rlcs}.}\label{tab:cpu-target}
\end{table*}

\section{Conclusions}\label{sec6}
This paper introduces MARS-ADMM, a single-loop stochastic algorithm designed for nonsmooth composite optimization on compact Riemannian submanifolds. 
MARS-ADMM combines a prediction-correction Gauss-Seidel sweep, a self-normalized Riemannian gradient stepsize, and a damped dual update. 
Our analysis reveals that the damped multiplier step behaves like a contraction mapping toward the dual maximizer of a Moreau envelope. It also shows that MARS-ADMM achieves an $\epsilon$-KKT point with oracle complexity bounded by $\mathcal O(\epsilon^{-3})$. This improves upon the best-known rate for stochastic Riemannian primal-dual methods and matches deterministic nonsmooth Riemannian bounds. 
Empirical simulation results on sparse principal component analysis and regularized classification over the sphere validate the effectiveness of our proposed approach and demonstrate its promising performance.

\bibliographystyle{siamplain}
\bibliography{references}

\end{document}